\documentclass[12pt]{amsart}

\usepackage{mathrsfs}
\usepackage{bbm}%
\usepackage{bbding}
\usepackage[mathcal]{euscript}
\usepackage{graphicx}
\usepackage{amsthm,amscd}
\usepackage{calc}
\usepackage{mathrsfs,dsfont}
\usepackage{CJK,fancyhdr,amscd}
\usepackage{anysize}
\usepackage{amsmath,amssymb,amsfonts}
\usepackage{epsfig,enumerate}
\usepackage{latexsym}
\usepackage{indentfirst, latexsym}
\usepackage{graphics}
\usepackage[all,poly,knot]{xy}
\usepackage[pagebackref]{hyperref}

\renewcommand{\baselinestretch}{1}
\numberwithin{equation}{section}
\def\Re{{\rm Re}}
\def\Im{{\rm Im}}
\def\p{\partial}
\def\o{\overline}
\def\b{\bar}
\def\mb{\mathbb}
\def\mc{\mathcal}
\def\n{\nabla}
\def\v{\varphi}
\def\w{\wedge}
\def\m{\omega}
\def\t{\triangle}
\def\wt{\widetilde}
\def\mf{\mathfrak}
\def\u{\underline}
\def\l{\lrcorner}
\theoremstyle{plain}
\newtheorem{thm}{Theorem}[section]
\newtheorem{lemma}[thm]{Lemma}
\newtheorem{prop}[thm]{Proposition}
\newtheorem{cor}[thm]{Corollary}
\theoremstyle{definition}
\newtheorem{rem}[thm]{Remark}

\newtheorem{ex}[thm]{Example}
\theoremstyle{definition}
\newtheorem{defn}[thm]{Definition}
\newcommand{\comment}[1]{}
\def\baselinestretch{1}
\usepackage{fancyhdr}
\begin{document}

\title{Geometry of logarithmic forms and deformations of complex structures}
\makeatletter
\let\uppercasenonmath\@gobble
\let\MakeUppercase\relax
\let\scshape\relax
\makeatother

\author[Kefeng Liu]{Kefeng Liu}
\author[Sheng Rao]{Sheng Rao}
\author[Xueyuan Wan]{Xueyuan Wan}

\address{Kefeng Liu, Department of Mathematics, Capital Normal University, Beijing, 100048, China; Department of Mathematics, University of California at Los Angeles, California 90095, USA}

\email{liu@math.ucla.edu}

\address{Sheng Rao, School of Mathematics and Statistics, Wuhan  University,
Wuhan 430072, China.}
\email{likeanyone@whu.edu.cn}

\address{Xueyuan Wan, Mathematical Sciences, Chalmers University of Technology, 412 96 Gothenburg, Sweden.}
\email{xwan@chalmers.se}

\thanks{Liu is partially supported by NSF (Grant No. 1510216); Rao is partially supported by NSFC (Grant No. 11671305, 11771339);
 Wan is partially supported by China Scholarship Council/University of California, Los Angeles Joint PhD. Student}

\subjclass[2010]{Primary 58A14; Secondary 18G40, 58A10, 32G05, 14J32, 58A25}
\keywords{Hodge theory; Spectral sequences, hypercohomology, Differential forms, Deformations of complex structures, Calabi-Yau manifolds, Currents}

\begin{abstract}
We present a new method to solve certain $\bar\partial$-equations for logarithmic differential forms by using harmonic integral theory for currents on K\"ahler manifolds. The result can be considered as a $\partial\bar\partial$-lemma for logarithmic forms. As applications, we generalize the result of Deligne about closedness of logarithmic forms, give geometric and simpler proofs of Deligne's degeneracy theorem for the logarithmic Hodge to de Rham spectral sequences at $E_1$-level, as well as certain injectivity theorem on compact K\"ahler manifolds.

Furthermore, for a family of logarithmic deformations of complex structures on K\"ahler manifolds, we construct the extension for any logarithmic $(n,q)$-form on the central fiber and thus deduce the local stability of log Calabi-Yau structure by extending an iteration method to the logarithmic forms. Finally we prove the unobstructedness of the deformations of a log Calabi-Yau pair and a pair on a Calabi-Yau manifold by differential geometric method.
 \end{abstract}
\maketitle
\tableofcontents

\section*{Introduction} \label{s0}

The basic theory on sheaf of logarithmic differential forms and of sheaves with logarithmic integrable connections over smooth projective manifolds were developed by P. Deligne in \cite{Del70}. H. Esnault and E. Viehweg investigated in \cite{EV06, Viehweg} the relations between logarithmic de Rham complexes and vanishing theorems on complex algebraic manifolds, and showed that many vanishing theorems follow from Deligne's degeneracy of logarithmic Hodge to de Rham spectral sequences at $E_1$-level. In \cite{Wan}, C. Huang, X. Yang, the first and third authors developed an effective analytic method to prove vanishing theorems for sheaves of logarithmic differential forms on compact K\"ahler manifolds. In this paper, the authors will present an effective differential geometric approach to the geometry of logarithmic differential forms which is used to study degeneracy of spectral sequences \cite{Del71}, injectivity theorems \cite{Ambro,Fujinoinj} in algebraic geometry and logarithmic deformations of complex structures \cite{KKP08,Iacono}.

Throughout this paper, let $X$ be an $n$-dimensional compact K\"ahler manifold and $D$ a simple normal crossing divisor on $X$. For any logarithmic $(p,q)$-form $$\alpha\in A^{0,q}(X,\Omega^p_X(\log D)),$$ as described in Section \ref{section1}, with $\b{\p}\p\alpha=0$, we will present a method to solve the $\b{\p}$-equation
 \begin{align}\label{0.1}
 \b{\p}x=\p\alpha
 \end{align}
such that $x\in A^{0,q-1}(X,\Omega^{p+1}_X(\log D))$. For the case $D=\emptyset$, the equation (\ref{0.1}) is easily solved by the $\partial\bar{\partial}$-lemma  from standard Hodge theory on forms, as discussed in \cite[p. 84]{Griffith}. For the case $D\neq \emptyset$, the equation (\ref{0.1}) is defined on the open manifold $X-D$, and the naive possible approach is to use the $L^2$-Hodge theory with respect to some complete K\"ahler metric as in \cite{Wan}. However, a logarithmic form is not necessarily  $L^2$-integrable and so one needs some new methods to solve this equation (\ref{0.1}).  Inspired by the work \cite{Nog} of J. Noguchi, one may consider a logarithmic form as a current on $X$, where the harmonic integral theory \cite{de,Kodaira} by G. de Rham and K. Kodaira is available.

Roughly speaking, the current $T_{\p\alpha}$ associated to $\p\alpha$ can be decomposed into two terms, one term of residue and the other one in the image of $\b{\p}$ as shown in (\ref{2.7}).  An iteration trick shows that the residue term also lies in the image of $\b{\p}$  when acting on a smooth differential form vanishing on $D$ as shown in Lemma \ref{lemma2} and so does $T_{\p\alpha}$. Notice that the sheaf of logarithmic differential forms is locally free and thus the logarithmic form $\p\alpha$ can also be viewed as a bundle-valued smooth differential form.  By these and the bundle-valued Hodge decomposition theorem, we  prove our first main theorem which can be considered as a $\p\b{\p}$-lemma for logarithmic forms.
 \begin{thm}[=Theorem \ref{thm2}]\label{main theorem}
 Let $X$ be a compact K\"ahler manifold and $D$ a simple normal crossing divisor on $X$. For any $\alpha\in A^{0,q}(X,\Omega^p_X(\log D))$ with $\b{\p}\p\alpha=0$, there exists a solution $$x\in A^{0,q-1}(X,\Omega^{p+1}_X(\log D))$$ for the $\b{\p}$-equation $(\ref{0.1})$.
 \end{thm}
Similar to Theorem \ref{main theorem}, one is able to obtain the second main theorem:
\begin{thm}[=Theorem \ref{sol-CY}]\label{0sol-CY} With the same notations as in Theorem \ref{main theorem},
if $$\alpha\in A^{n,n-q}(X, T_X^p(-\log D))\subset A^{n-p,n-q}(X)$$ with $\b{\p}\p\alpha=0$, then there is a solution $x\in A^{n,n-q-1}(X, T_X^{p-1}(-\log D))$ such that
$$
\b{\p}x=\p\alpha.	
$$
\end{thm}
Here one identifies an element of $A^{n,n-q}(X,T^p_X(-\log D))$ with an $(n-p,n-q)$-form by contraction of the $(n,n-q)$-form with the $T^p_X(-\log D)$-valued coefficient as given explicitly in  (\ref{2.3}).

Then we present three kinds of applications of Theorems $\ref{main theorem}$ and \ref{0sol-CY} to algebraic geometry.
As the first application of Theorem \ref{main theorem}, we generalize one result of Deligne on $d$-closedness of logarithmic forms on a smooth complex quasi-projective variety.

\begin{cor}[=Corollary \ref{g-No}]\label{closedness} With the same notations as in Theorem \ref{main theorem},
if $\alpha\in A^{0,0}(X,\Omega^p_X(\log D))$ with $\b{\p}\p\alpha=0$, then $\p\alpha=0$.	
\end{cor}

It is well-known that Deligne's degeneracy of logarithmic Hodge to de Rham spectral sequences at $E_1$-level \cite{Del71} is a fundamental result and has great impact in algebraic geometry, such as vanishing and injectivity theorems. P. Deligne and L. Illusie \cite{Del87} also proved this degeneracy by using a purely algebraic positive characteristic method.
For compact K\"ahler manifolds, as the second application of Theorem \ref{main theorem}, we can give a geometric and simpler proof of Deligne's degeneracy theorem.
\begin{thm}[=Theorem \ref{Dss}]\label{0Deligne} With the same notations as in Theorem \ref{main theorem},
the logarithmic Hodge to de Rham spectral sequence associated with the Hodge filtration
$$
	E^{p,q}_1=H^q(X,\Omega^p_X(\log D))\Rightarrow \mb{H}^{p+q}(X, \Omega^*_X(\log D))	
$$
degenerates at the $E_1$-level.
\end{thm}

As a direct corollary of the above theorem,
\begin{equation*}
  \label{cor0.1}
\dim_\mb{C}H^k(X-D,\mb{C})=\sum_{p+q=k}\dim_\mb{C} H^q(X,\Omega^p_X(\log D)).
\end{equation*}	

Similar to Theorem \ref{0Deligne}, Theorem \ref{0sol-CY} gives rise to a dual version of Theorem \ref{0Deligne}.
This duality appears in \cite[Remark 2.11]{Fujinosjv}.
\begin{cor}[=Corollary \ref{Ddual}]\label{ss2}  With the same notations as in Theorem \ref{main theorem},
the spectral sequence associated with the Hodge filtration
\begin{align*}
E^{p,q}_1=H^q(X,\Omega^p_X(\log D)\otimes \mc{O}_X(-D))\Rightarrow 	\mb{H}^{p+q}(X,\Omega^*_X(\log D)\otimes \mc{O}_X(-D))
\end{align*}
	degenerates at $E_1$-level.
\end{cor}

The third result of Theorem \ref{main theorem} is an injectivity theorem for compact K\"ahler manifolds, whose algebraic version was first proved by F. Ambro \cite[Theorem 2.1]{Ambro} and the equivalence in the statement was proposed in \cite[Remark 2.6 and Corollary 2.7]{Ambro}.
\begin{cor}[=Corollary \ref{inj}] With the same notations as in Theorem \ref{main theorem}, the restriction homomorphism
\begin{align*}
	H^q(X,\Omega^n_X(\log D))\xrightarrow{i} H^q(U, K_U)
\end{align*}
	is injective, where $U=X-D$. Equivalently, if $\Delta$ is an effective divisor with $\text{Supp}(\Delta)\subset \text{Supp}(D)$, then
 the natural homomorphism induced by the inclusion $\mc{O}_X\subset \mc{O}_X(\Delta)$
	\begin{align*}
		H^q(X,\Omega^n_X(\log D))\xrightarrow{i'} H^q(X, \Omega^n_X(\log D)\otimes\mc{O}_X(\Delta))
	\end{align*}
is injective.
\end{cor}

Note that Ambro's algebraic version was generalized by O. Fujino \cite[Theorem 1.1]{Fujinoinj} on a simple normal crossing algebraic variety.

Finally, we describe another three applications of Theorems \ref{main theorem} and \ref{0sol-CY} to logarithmic deformations of complex structures \cite{KKP08,Iacono}.
The local stability of certain geometric structure under deformation is an interesting topic in deformation theory of complex structures, in which the power series method, initiated by Kodaira-Nirenberg-Spencer and Kuranishi \cite{MK,k}, plays a prominent role there.  In \cite{RwZ}, Q. Zhao, the second and third authors presented a power series proof for the classical Kodaira-Spencer's local stabilities of K\"ahler structures. From \cite{Liu}, one can construct a smooth $d$-closed extension for any holomorphic $(n,0)$-form on the central fiber. Inspired by these results, we can consider the problems of $d$-closed extension of logarithmic forms under logarithmic deformations.  The logarithmic deformation of a pair $(X,D)$ is a special deformation of $X$ such that $\cup_{t\in S}D_t$ is a closed analytic subset in $\cup_{t\in S}X_t$ as in Definition \ref{kwawa1}, which is defined and developed in \cite{kawa, kawa1}, while the log Hodge structure theory is also developed in \cite{Kato}.

Let $\bar{\partial}_t$ denote the $\bar{\partial}$-operator on $X_t$. As the first application of Theorem \ref{main theorem} to deformation theory, we have
\begin{thm}[=Theorem \ref{thm3}]\label{main theorem 1} With the same notations as in Theorem \ref{main theorem},
for any logarithmic deformations $(X_t,D_t), t\in S$ of the pair $(X,D)$ with $X_{0}=X$, induced by the Beltrami differential $\varphi:=\varphi(t)\in A^{0,1}(X, T_X(-\log D))$,  and any $\b{\p}$-closed logarithmic $(n,q)$-form $\Omega$ on the central fiber $X$, there exists a small neighborhood $\Delta\subset S$ of $0$ and a smooth family $\Omega(t)$ of logarithmic $(n,q)$-form on the central fiber $X$, such that
$$
e^{i_\varphi}(\Omega(t))\in A^{0,q}(X_t,\Omega^n_{X_t}(\log D_t))
$$
which is $\b{\p}_t$-closed on $X_t$ for any $t\in \Delta$ and satisfies $(e^{i_\varphi}(\Omega(t)))(0)=\Omega$.
\end{thm}
By definition, a \textit{log Calabi-Yau pair} is a pair $(X,D)$ such that  $D$ is a simple normal crossing divisor on an $n$-dimensional K\"ahler manifold $X$ and the logarithmic canonical line bundle $\Omega^n_X(\log D)\cong \mc{O}_X(K_X+D)$ is trivial. As a direct corollary of Theorem \ref{main theorem 1}, one obtains the local stabilities of log Calabi-Yau structures under deformations.

Now we discuss other two applications to logarithmic deformations.
On projective manifolds, L. Katzarkov, M. Kontsevich, and T. Pantev \cite{KKP08} proved the unobstructedness of logarithmic deformations of a log Calabi-Yau pair and a pair on a Calabi-Yau manifold. More precisely, they used Dolbeault type complexes to construct a differential Batalin-Vilkovisky algebra such that the associated differential graded Lie algebra (DGLA) controls the deformation problem. If the differential Batalin-Vilkovisky algebra has a degeneracy property then the associated DGLA is homotopy abelian. Using the notion of the Cartan homotopy, D. Iacono \cite{Iacono} obtained an alternative proof  of the unobstructedness theorems. Both of their proofs rely on the degeneracy of spectral sequences in Theorem \ref{0Deligne} and Corollary \ref{ss2}. Here we present a differential geometric proof which has potential applications to extension problems.

Combining the methods originally from \cite{T87,To89,LSY} and developed in \cite{Liu,RZ,RZ2, RZ15, RwZ}, we use Theorem \ref{main theorem} to prove:
\begin{thm}[=Theorem \ref{Bel-equ}]\label{0CYlog}
Let $(X,D)$ be a log Calabi-Yau pair and $[\varphi_1]\in H^{0,1}(X, T_X(-\log D))$. Then on a small disk of $0$ in $\mathbb{C}^{\dim_{\mathbb{C}} H^{0,1}(X, T_X(-\log D))}$, there exists a holomorphic family $$\varphi(t)\in A^{0,1}(X, T_X(-\log D)),$$  such that
$$
\b{\p}\varphi(t)=\frac{1}{2}[\varphi(t),\varphi(t)],  \quad \frac{\p \varphi}{\p t}(0)=\varphi_1.	
$$
\end{thm}

Recall that a compact $n$-dimensional K\"ahler manifold $X$ is called a \emph{Calabi-Yau manifold} if it admits a nowhere vanishing holomorphic $(n,0)$-form.
Similar to Theorem \ref{0CYlog}, by use of Theorem \ref{0sol-CY}, one obtains another unobstructedness theorem for logarithmic deformation.
\begin{thm}[=Theorem \ref{log}]\label{0log}
Let $X$ be a Calabi-Yau manifold and $D$ a simple normal crossing divisor on $X$. Then
the pair $(X,D)$ has unobstructed  logarithmic deformations.
\end{thm}

This article is organized as follows. In Section \ref{section1}, we introduce the definitions and basic properties of sheaves of logarithmic differential forms, Poincar\'e residues and currents, and describe Kodaira and de Rham's Hodge Theorem in the sense of currents. In Section \ref{section2}, we will prove main Theorems \ref{main theorem} and \ref{0sol-CY}. We present the applications of two main theorems mentioned above to algebraic geometry and logarithmic deformation unobstructedness theorems in Sections \ref{section3} and \ref{section4}, respectively.

\textbf{Acknowledgement}:
The second author would like to express his
gratitude to Professor J.-P. Demailly, Dr. Ya Deng, Long Li and Jian Wang
of Institut Fourier, Universit\'{e} Grenoble Alpes, where this paper was completed and
there were many stimulating discussions on it and other various aspects of complex geometry.
We express our sincere gratitude to Dr. Jie Tu and Professor Shin-ichi Matsumura for their careful reading our
manuscript, many useful comments and pointing out the much related paper \cite{ma}, and also Professors F. Ambro and O. Fujino
for their interest in this paper.

\section{Logarithmic forms and currents}\label{section1}
In this section, we introduce some basic facts and notations on the sheaf of logarithmic forms, Poincar\'e residue and currents, to be used throughout this paper. For more details, one may refer to \cite{Dem, Viehweg, Griffith, kawa, Kodaira, Nog}.

Let $(X,\omega)$ be a compact K\"ahler manifold of dimension $n$, and let $D$ be a simple normal crossing divisor on it, i.e., $D=\sum_{i=1}^r D_i$, where  the $D_i$ are distinct smooth hypersurfaces intersecting transversely in $X$.

Denote by $\tau: Y=X-D\to X$ the natural inclusion and
$$
\Omega^p_X(*D)=\lim_{\rightarrow\atop \nu}\Omega^p_X(\nu\cdot D)=\tau_*\Omega^p_Y.
$$
Then $(\Omega^{\bullet}_X(*D),d)$ is a complex. The sheaf of logarithmic forms $$\Omega^p_X(\log D)$$ (introduced by Deligne in \cite{De}) is defined as the subsheaf of $\Omega^p_X(*D)$ with logarithmic poles along $D$, i.e., for any open subset $V\subset X$,
$$\Gamma(V,\Omega^p_X(\log D))=\{ \alpha\in \Gamma(V,\Omega^p_X(*D)): \alpha \,\,\text{and}\,\, d\alpha\,\, \text{have simple poles along}\,\, D\}.$$
 From (\cite[II, 3.1-3.7]{Del70} or \cite[Properties 2.2]{Viehweg}), the log complex $(\Omega^{\bullet}_X(\log D), d)$ is a subcomplex of $(\Omega^{\bullet}_X(*D), d)$ and $\Omega^p_X(\log D)$ is locally free,
 $$\Omega^p_X(\log D)=\wedge^p\Omega^1_X(\log D).$$

 For any $z\in X$, which $k$ of these $D_i$ pass, we may choose local holomorphic coordinates $\{z^1,\cdots, z^n\}$ in a small neighborhood $U$ of $z=(0,\cdots, 0)$ such that
 $$D\cap U=\{z^1\cdots z^k=0\}$$
 is the union of coordinates hyperplanes. Such a pair
\begin{equation}\label{lcs}
(U,\{z^1,\cdots, z^n\})
\end{equation}
is called a \textit{logarithmic coordinate system} \cite[Definition 1]{kawa}.
 Then $\Omega^p_X(\log D)$ is generated by the holomorphic forms and logarithmic differentials $dz^i/z^i$ ($i=1,\ldots, k$), i.e.,
 \begin{align}\label{1.6}
 \Omega^p_X(\log D)=\Omega^p_X\left\{\frac{dz^1}{z^1},\cdots,\frac{dz^k}{z^k}\right\}.	
 \end{align}
Denote by $$A^{0,q}(X,\Omega^p_X(\log D))$$ the space of smooth $(0,q)$-forms on $X$ with values in $\Omega^p_X(\log D)$, and call an element of $A^{0,q}(X,\Omega^p_X(\log D))$  a \emph{logarithmic $(p,q)$-form}.

For any $\alpha\in A^{0,q}(X,\Omega^p_X(\log D))$, we can write
\begin{align}\label{expression of alpha}
\alpha=\alpha_1+\frac{dz^1}{z^1}\wedge \alpha_2	
\end{align}
on $U$, where $\alpha_1$ does not contain $dz^1$ and $\alpha_2\in A^{0,q}(U,\Omega^{p-1}_X(\log \sum_{i=2}^rD_i))$. Denoting by
$$\iota_{D_i}:D_i\hookrightarrow X$$
the natural inclusion, without loss of generality, we may assume that $$\{z^1=0\}=D_1\cap U,$$ and put
$$\label{defnRes}
\text{Res}_{D_1}(\alpha)=\iota^*_{D_1}(\alpha_2)	
$$
on $D_1\cap U$. Then $\text{Res}_{D_1}(\alpha)$ is globally well-defined and
$$
\text{Res}_{D_1}(\alpha)\in A^{0,q}(D_1,\Omega^{p-1}_{D_1}(\log \sum_{i=2}^r D_i\cap D_1)).	
$$
 Set the so-called the \emph{Poincar\'e residue} (cf. \cite[\S 2]{Nog}) as
$$
 \text{Res}(\alpha)=\sum_{i=1}^r\text{Res}_{D_i}(\alpha).
$$
And we also define
$$
\text{Res}_{D_{i_1}\cdots D_{i_l}}(\alpha)=\text{Res}_{D_{i_1}\cap\cdots \cap D_{i_l}}\circ \cdots\circ\text{Res}_{D_{i_1}\cap D_{i_2}}\circ \text{Res}_{D_{i_1}}(\alpha),
$$
which lies in $$A^{0,q}(D_{i_1}\cap\cdots \cap D_{i_l}, \Omega^{p-l}_{D_{i_1}\cap\cdots \cap D_{i_l}}(\log \sum_{j\neq \{i_1,\ldots, i_l\}}D_j\cap D_{i_1}\cap\cdots \cap D_{i_l})).$$  From \cite[(2.3)]{Nog}, one has
\begin{align}\label{Reszero}
	\text{Res}_{D_{i_1}\cdots D_{i_{j}} D_{i_{j+1}}\cdots D_{i_l}}(\alpha)+\text{Res}_{D_{i_1}\cdots D_{i_{j+1}} D_{i_{j}}\cdots D_{i_l}}(\alpha)=0.
\end{align}
We will consider $\text{Res}(\alpha)$ as a current of bidegree $(p,q+1)$
\begin{align}\label{2.2}
\text{Res}(\alpha)(\beta)=\sum_{i=1}^r\int_{D_i}\text{Res}_{D_i}(\alpha)\wedge \iota^*_{D_i}(\beta)	
\end{align}
for any smooth $(n-p,n-q-1)$-form $\beta$ on $X$.

Recall that a \textit{current} of bidegree $(p,q)$ on a  compact K\"ahler or generally complex manifold $X$ is a differential $(p,q)$-form with distribution coefficients. We refer the readers to \cite[Chapter I.2 and Chapter III]{Dem} for a comprehensive introduction to current theory. The space of currents of bidegree $(p,q)$ over $X$ will be denoted by $\mc{D}'^{p,q}(X)$ and it is topologically dual to the space $A^{n-p,n-q}(X)$ of smooth differential forms of bidegree $(n-p,n-q)$. There are two classical examples concerning currents, which are very useful later.

\begin{ex}\label{ex1}
\begin{itemize}
	\item[(1)] Let $S\subset X$ be a closed $p$-dimensional complex submanifold with the canonical orientation. Then the integral over $S$ gives a $(p,p)$-bidimensional current, denoted by $T_S$, as
$$
	T_S(\alpha)=\int_S \alpha,\quad \alpha\in A^{p,p}(X)
$$
since each $(r,s)$-form of total degree $r+s=2p$ has zero restriction to $Z$ unless $(r,s)=(p,p)$.
\item[(2)] For any complex differential $(p,q)$-form $\alpha$ with $L_{loc}^1$ coefficients on $X$, there is a current $T_\alpha$  associated with $\alpha$, such that
for any continuous $(n-p,n-q)$-form $\beta$ on $X$
$$
T_{\alpha}(\beta)=\int_X \alpha\wedge\beta.	
$$
\end{itemize}	
\end{ex}

By Example \ref{ex1}, we denote by $T_{\text{Res}_{D_i}(\alpha)}$ the current on $D_i$ associated with $\text{Res}_{D_i}(\alpha)$, and then rewrite (\ref{2.2}) as
\begin{align}\label{Rescurrent}
	\text{Res}(\alpha)(\beta)=\sum_{i=1}^r T_{\text{Res}_{D_i}(\alpha)}(\iota^*_{D_i}(\beta)).
\end{align}

For any $(p,q)$-current $T$, the exterior derivative $dT$ and the adjoint $*T$ are defined by (cf. \cite{Kodaira})
$$
(dT)(\alpha)=(-1)^{p+q+1}T(d\alpha),\quad (*T)(\alpha)=(-1)^{p+q}T(*\alpha),	
$$
where the star operator $*$ is defined by
$$
\alpha\wedge *\beta=\langle\alpha,\beta\rangle\frac{\omega^n}{n!}	
$$
for any smooth $(p,q)$-forms $\alpha,\beta$, and the inner $\langle\cdot,\cdot\rangle$ is induced by  $\omega$ on the space $A^{p,q}(X)$ of $(p,q)$-forms.

Set $d^*=-*d*$, $\Delta=dd^*+d^*d$. Then:
 \begin{thm}[\cite{de}]\label{thm1}
 There exists one (and only one) linear operator $\mb{G}$ mapping any $(p,q)$-current $T$ into a $(p,q)$-current $\mb{G}T$ which has the following properties:
$$
 \Delta\mb{G}T=\mb{G}\Delta T=T-\mb{H}T, \quad \mb{H}\mb{G}T=\mb{G}\mb{H}T=0, 	
 $$
	where $\mb{H}$ is the harmonic projection, defined by
	$$\mb{H}T=\sum_{k=1}^{N}T(*e_k)e_k,$$
with $N=\dim_{\mathbb{C}} H^{p,q}(X,\mb{C})$ and $\{e_k\}_{k=1}^N$ the orthonormal harmonic $(p,q)$-forms.
 \end{thm}

From the definitions of $\Delta$ and $\mb{H}$, one obtains
\begin{align}\label{1.4}
(\Delta T)(\alpha)=T(\Delta\alpha), \quad (\mb{H}T)(\alpha)=T(\mb{H}\alpha)	
\end{align}
for any $\alpha\in A^{n-p,n-q}(X)$. By Theorem \ref{thm1}, we have
$$
(\mb{G}T)(\alpha)=(\mb{G}T)(\Delta\mb{G}\alpha+\mb{H}\alpha)=(\Delta\mb{G}T)(\mb{G}\alpha)=T(\mb{G}\alpha).	
$$

One can also define operators $\p$, $\b{\p}$, $\p^*$ and $\b{\p}^*$ acting on a $(p,q)$-current $T$ by a similar way,
\begin{align}\label{1.1}
(\p T)(\alpha)=(-1)^{p+q+1}T(\p\alpha),\quad (\b{\p}T)(\alpha)=(-1)^{p+q+1}T(\b{\p}\alpha)	
\end{align}
and
\begin{align}\label{1.2}
(\p^*T)(\alpha)=(-1)^{p+q}T(\p^*\alpha),\quad (\b{\p}^*T)(\alpha)=(-1)^{p+q}T(\b{\p}^*\alpha)	
\end{align}
for any smooth $\alpha$. Setting $\Delta'=\p^*\p+\p\p^*$ and $\Delta''=\b{\p}^*\b{\p}+\b{\p}\b{\p}^*$, one deduces from (\ref{1.1}) and (\ref{1.2}) that
\begin{align}\label{1.5}
(\Delta'T)(\alpha)=T(\Delta'\alpha),\quad (\Delta''T)(\alpha)=T(\Delta''\alpha).	
\end{align}
Note that $\Delta'=\Delta''=\frac{1}{2}\Delta$ on smooth forms since $(X,\omega)$ is a K\"ahler manifold, and by (\ref{1.4}) and (\ref{1.5}),
\begin{align}\label{1.3}
\Delta'T=\Delta''T=\frac{1}{2}\Delta T.	
\end{align}
If one sets $\mb{G}'=\mb{G}''=2\mb{G}$, then it follows from (\ref{1.3}) and Theorem \ref{thm1} that
\begin{align}\label{1.7}
\Delta'\mb{G}'T=\Delta''\mb{G}''T=T-\mb{H}T.	
\end{align}

\section{Solving two $\b{\p}$-equations for logarithmic forms} \label{section2}

In this section, we will solve two $\b{\p}$-equations in terms of logarithmic forms on a compact K\"ahler manifold, which is the starting point of this paper.

Given a compact K\"ahler manifold $X$ and a simple normal crossing divisor $$D=\sum_{i=1}^r D_i$$ on it, we first consider the equation
\begin{align}\label{main equation}
\b{\p}x=\p\alpha	
\end{align}
for any $\alpha\in A^{0,q}(X,\Omega^p_X(\log D))$ with $\b{\p}\p\alpha=0$.

According to \cite{Nog} for example, for any $\alpha\in A^{0,q}(X,\Omega^p_X(\log D))$, there is a current  $T_\alpha\in \mc{D}'^{p,q}(X)$ associated with it, which is defined by
$$
T_{\alpha}(\beta)=\int_X \alpha\wedge \beta,\quad \beta\in A^{n-p,n-q}(X).	$$
 Under the logarithmic coordinate system \eqref{lcs}, one may assume that $D_i=\{z^i=0\}$ locally and
let $$(D_i)_{\epsilon}:=\{|z^i|<\epsilon\}$$ be the $\epsilon$-tubular neighborhood of $D_i$. For any $\beta\in A^{n-p-1,n-q}(X)$, one has
\begin{align}\label{2.1}
\begin{split}
	(\p T_{\alpha}-T_{\p\alpha})(\beta)&=\int_X ((-1)^{p+q+1}\alpha\wedge\p\beta-\p\alpha\wedge\beta)\\
	&=-\int_X\p(\alpha\wedge\beta)\\
    &=-\int_X d(\alpha\wedge \beta)\\
	&=-\lim_{\epsilon\to 0}\int_{X-\cup_{i=1}^r(D_i)_{\epsilon}}d(\alpha\wedge\beta)\\
	&=\lim_{\epsilon\to 0}\sum_{i=1}^r \int_{\p(D_i)_{\epsilon}-\cup_{i\neq j}^r(D_j)_{\epsilon}}\alpha\wedge \beta\\
	&=2\pi\sqrt{-1}\sum_{i=1}^r\int_{D_i}\text{Res}_{D_i}(\alpha)\wedge \iota^*_{D_i}(\beta)\\
    &=0,
	\end{split}
\end{align}
where the last equality holds since the degree of $\text{Res}_{D_i}(\alpha)\wedge \iota^*_{D_i}(\beta)$ is $(n-2,n)$, not compatible with the dimension of $D_i$, and the last but one equality follows from a polar coordinate calculation. In fact, by the expression of (\ref{expression of alpha}), one has
\begin{align*}
	\lim_{\epsilon\to 0}\int_{\p (D_1)_{\epsilon}-\cup_{i\neq 1}^r(D_i)_{\epsilon}}\alpha\wedge \beta &=\lim_{\epsilon\to 0}\int_{\p (D_1)_{\epsilon}-\cup_{i\neq 1}^r(D_i)_{\epsilon}}\frac{dz^1}{z^1}\wedge \alpha_2\wedge \beta\\
	&=\lim_{\epsilon\to 0}\int_{D_1-\cup_{i\neq 1}^r(D_i)_{\epsilon}}\left(\int_{|z^1|=\epsilon}\frac{dz^1}{z^1}\right)\iota^*_{D_1}(\alpha_2\wedge \beta)\\
	&=2\pi\sqrt{-1}\int_{D_1}\text{Res}_{D_1}(\alpha)\wedge \iota^*_{D_1}(\beta).
\end{align*}	

By a similar computation to (\ref{2.1}), one has
$$
	dT_{\alpha}-T_{d\alpha}=2\pi\sqrt{-1}\text{Res}(\alpha).
$$
Combining it with (\ref{2.1}) gives
\begin{align}\label{2.6}
\b{\p}T_{\alpha}-T_{\b{\p}\alpha}=2\pi\sqrt{-1}\text{Res}(\alpha).	
\end{align}
These calculations are inspired by \cite[Formula (2.2)]{Nog}.

Let $$E^p$$ be the holomorphic vector bundle on $X$ associated with the locally free sheaf $\Omega^p_X(\log D)$. Then there exists an isomorphism $\mathcal{I}$ between the spaces of logarithmic $(p,q)$-forms and $E^p$-valued $(0,q)$-forms,
\begin{equation}\label{iso1}
\mathcal{I}: A^{0,q}(X,\Omega^p_X(\log D))\to A^{0,q}(X, E^p),
\end{equation}
and its dual map
\begin{equation}\label{iso2}
 \mathcal{I}^*: A^{n,n-q}(X, (E^p)^*)\to A^{n,n-q}(X,T^p_X(-\log D))\subset A^{n-p,n-q}(X),
\end{equation}
where $T^p_X(-\log D)$ is the \textit{logarithmic tangent sheaf} which is the dual sheaf of $\Omega^p_X(\log D)$ (cf. \cite{Saito}). And we will identify an element of $A^{n,n-q}(X,T^p_X(-\log D))$ with an $(n-p,n-q)$-form by contraction of the $(n,n-q)$-form with the $T^p_X(-\log D)$-valued coefficient, see (\ref{2.3}) for precise definition.

By the construction \eqref{iso2} of $\mathcal{I}^*$, we have
\begin{lemma}\label{lemma1}
For any $\beta\in A^{n,n-p}(X, (E^p)^*)$, one has
$$
\iota^*_{D_i}\mathcal{I}^*(\beta)=0,\ i=1,\cdots,r.	
$$
\end{lemma}
\begin{proof}
Without loss of generality, we may assume that $D_i\cap U=\{z^i=0\}$ for some small open set $U\subset X$. By the definition of $\mathcal{I}^*$, $\mathcal{I}^*(\beta)$ contains either $z^i$ or $dz^i$. Since $\iota^*_{D_i}z^i=0$ and $\iota^*_{D_i}(dz^i)=d(\iota^*_{D_i}z^i)=0$, $\iota^*_{D_i}\mathcal{I}^*(\beta)=0$.
\end{proof}

The two mappings $\mathcal{I}$ and $\mathcal{I}^*$ in \eqref{iso1} and \eqref{iso2} are related by:
\begin{lemma}\label{lemma3}
For any $\alpha\in A^{0,q}(X,\Omega^p_X(\log D))$ and $\beta\in A^{n,n-q}(X, (E^p)^*)$, one has	
$$
T_{\alpha}(\mathcal{I}^*(\beta))=\int_X \mathcal{I}(\alpha)(\beta),	
$$
where $\mathcal{I}(\alpha)(\beta)$ is an $(n,n)$-form, obtained by pairing the values of $\mathcal{I}(\alpha)$ with $\beta$.
\end{lemma}
\begin{proof}
 From (\ref{1.6}), we may assume locally that
$$
\alpha=\alpha_{i_1\cdots i_p\b{j_1}\cdots\b{j_q}} d\b{z}^{j_1}\wedge \cdots\wedge d\b{z}^{j_q}\wedge\frac{dz^{i_1}}{z^{i_1}}\wedge\cdots\wedge \frac{dz^{i_p}}{z^{i_p}}
$$
	and
$$\beta=\beta^{i_1\cdots i_p}_{\b{j_1}\cdots\b{j}_{n-q}}d\b{z}^{j_1}\wedge \cdots\wedge d\b{z}^{j_{n-q}}\wedge dz^1\wedge\cdots\wedge dz^n\otimes (\mathcal{I}^*)^{-1}(z^{i_1}\frac{\p\ }{\p z^{i_1}}\wedge \cdots \wedge z^{i_p}\frac{\p\ }{\p z^{i_p}}),
$$
where $(\mathcal{I}^*)^{-1}$ is used to denote the isomorphism between $T^p_X(-\log D)$ and $(E^p)^*$. By the definition \eqref{iso2} of $\mathcal{I}^*$, one has
\begin{align}\label{2.3}
	\mathcal{I}^*(\beta)=\beta^{i_1\cdots i_p}_{\b{j_1}\cdots\b{j}_{n-q}}z^{i_1}\cdots z^{i_p}i_{\frac{\p\ }{\p z^{i_p}}}\circ \cdots\circ i_{\frac{\p\ }{\p z^{i_1}}}(d\b{z}^{j_1}\wedge \cdots\wedge d\b{z}^{j_{n-q}}\wedge dz^1\wedge\cdots\wedge dz^n).
\end{align}
So
$$
\alpha\wedge \mathcal{I}^*(\beta)=\epsilon^{j_1\cdots j_ql_1\cdots l_{n-q}}_{1\cdots n}\epsilon^{i_1\cdots i_p}_{k_1\cdots k_p}\alpha_{i_1\cdots i_p\b{j_1}\cdots\b{j_q}}\beta^{k_1\cdots k_p}_{\b{l_1}\cdots\b{l}_{n-q}}d\b{z}^1\wedge\cdots\wedge d\b{z}^n\wedge dz^1\wedge\cdots\wedge dz^n,	
$$
where
$$
\epsilon^{i_1\cdots i_p}_{j_1\cdots j_p}=
\begin{cases}
	1, &i_1\cdots i_p\ \text{is an even permutation of }  j_1\cdots j_p,\\
-1, &i_1\cdots i_p \ \text{is an odd permutation of }  j_1\cdots j_p,\\
0, &\text{otherwise}.
\end{cases}	
$$

On the other hand, since
$$
\mathcal{I}(\alpha)=\alpha_{i_1\cdots i_p\b{j_1}\cdots\b{j_q}} d\b{z}^{j_1}\wedge \cdots\wedge d\b{z}^{j_q}\otimes\mathcal{I}\left(\frac{dz^{i_1}}{z^{i_1}}\wedge\cdots\wedge \frac{dz^{i_p}}{z^{i_p}}\right),
$$
$$\label{2.4}
\mathcal{I}(\alpha)(\beta)=\epsilon^{j_1\cdots j_ql_1\cdots l_{n-q}}_{1\cdots n}\epsilon^{i_1\cdots i_p}_{k_1\cdots k_p}\alpha_{i_1\cdots i_p\b{j_1}\cdots\b{j_q}}\beta^{k_1\cdots k_p}_{\b{l_1}\cdots\b{l}_{n-q}}d\b{z}^1\wedge\cdots\wedge d\b{z}^n\wedge dz^1\wedge\cdots\wedge dz^n=\alpha\wedge\mathcal{I}^*(\beta).	
$$
Therefore,
$$
T_{\alpha}(\mathcal{I}^*(\beta))=\int_X\alpha\wedge \mathcal{I}^*(\beta)=\int_X\mathcal{I}(\alpha)(\beta).	
$$
\end{proof}

Before solving the $\b{\p}$-equation (\ref{main equation}), we first give a key lemma.
\begin{lemma}\label{lemma2}
	For any $\alpha\in A^{0,q}(X,\Omega^p_X(\log D))$ with $\b{\p}\p\alpha=0$ and $\beta\in A^{n,n-q}(X, (E^{p+1})^*)$, one has for $l=1,\cdots,r$,
$$
	\sum_{i_1,\cdots, i_l=1}^r T_{\p\text{Res}_{D_{i_1}\cdots D_{i_l}}(\alpha)}(\iota^*_{i_1\cdots i_l}\b{\p}^*_{i_1\cdots i_{l-1}}\mb{G}''_{i_1\cdots i_{l-1}}\circ \cdots\circ \iota^*_{i_1}\b{\p}^*\mb{G}''(\mathcal{I}^*(\beta)))\in \text{Im}(\p\b{\p}\b{\p}^*\mb{G}'')(\mathcal{I}^*(\beta)),
$$
where $$\iota_{i_1\cdots i_l}: D_{i_1}\cap\cdots\cap D_{i_l}\to X$$ is the natural inclusion, and $\mb{G}''_{i_1\cdots i_{l-1}}$, $\b{\p}^*_{i_1\cdots i_{l-1}}$ are the operators on $$D_{i_1}\cap\cdots\cap D_{i_{l-1}}$$ with respect to the induced K\"ahler metric $\iota^*_{i_1\cdots i_{l-1}}\omega$ from the K\"ahler metric $\omega$ on $X$.
\end{lemma}
Here the notation $\eta\in \text{Im}(\p\b{\p}\b{\p}^*\mb{G}'')(\mathcal{I}^*(\beta))$ means that there exists some $(n-p-2,n-q)$-current $\tilde{T}$ on $X$ such that $\eta(\mathcal{I}^*(\beta))=\p\b{\p}\b{\p}^*\mb{G}''\tilde{T}(\mathcal{I}^*(\beta))$, where $\mathcal{I}^*$ is given by \eqref{iso2}.
\begin{proof}
Set $$A_l=\iota^*_{i_1\cdots i_l}\b{\p}^*_{i_1\cdots i_{l-1}}\mb{G}''_{i_1\cdots i_{l-1}}$$ and
$$
B_l=\sum_{i_1,\cdots, i_l=1}^r T_{\p\text{Res}_{D_{i_1}\cdots D_{i_l}}(\alpha)}(A_l\circ \cdots\circ A_1(\mathcal{I}^*(\beta))).
$$
By (\ref{1.7}), (\ref{2.1}) and $\p\circ\mb{H}=0$, one has
\begin{equation}\label{2.10}
	B_l=\sum_{i_1,\cdots, i_l=1}^r( \b{\p}^*_{i_1\cdots i_l}\b{\p}\mb{G}''_{i_1\cdots i_l}T_{\p\text{Res}_{D_{i_1}\cdots D_{i_l}}(\alpha)}+ \b{\p}\b{\p}^*_{i_1\cdots i_l}\mb{G}''_{i_1\cdots i_l}T_{\p\text{Res}_{D_{i_1}\cdots D_{i_l}}(\alpha)})(A_l\circ \cdots\circ A_1(\mathcal{I}^*(\beta))).
\end{equation}
Then Residue theorem \cite[Theorem 4.1.(ii) of Chapter II]{BHPV} tells us that
\begin{align}\label{2.11}
\p \text{Res}_{D_i}(\alpha)+\text{Res}_{D_i}(\p\alpha)=\b{\p} \text{Res}_{D_i}(\alpha)+\text{Res}_{D_i}(\b{\p}\alpha)=0,
\end{align}
where the expression \eqref{expression of alpha} of the logarithmic form is different from that in the reference with respect to the position of $\frac{dz^1}{z^1}$.

For the first term on the RHS of (\ref{2.10}),  the assumption $\b{\p}\p\alpha=0$, (\ref{2.11}),(\ref{Rescurrent}), (\ref{1.1}) and (\ref{1.2}) imply
\begin{align}\label{2.13}
\begin{split}
	&\quad \sum_{i_1,\cdots, i_l=1}^r( \b{\p}^*_{i_1\cdots i_l}\b{\p}\mb{G}''_{i_1\cdots i_l}T_{\p\text{Res}_{D_{i_1}\cdots D_{i_l}}(\alpha)})(A_l\circ \cdots\circ A_1(\mathcal{I}^*(\beta)))\\
	&=2\pi\sqrt{-1}\sum_{i_1,\cdots, i_l=1}^r(-1)^{p+q-l-1}\p\text{Res}(\text{Res}_{D_{i_1}\cdots D_{i_l}}(\alpha))(\b{\p}^*_{i_1\cdots i_l}\mb{G}''_{i_1\cdots i_l}\circ A_l\circ \cdots\circ A_1(\mathcal{I}^*(\beta)))\\
	&=2\pi\sqrt{-1}(-1)^{p+q-(l+1)}\sum_{i_1,\cdots, i_l, i_{l+1}=1}^r T_{\p\text{Res}_{D_{i_1}\cdots D_{i_{l+1}}}(\alpha)}(A_{l+1}\circ A_l\circ \cdots\circ A_1(\mathcal{I}^*(\beta)))\\
	&=2\pi\sqrt{-1}(-1)^{p+q-(l+1)}B_{l+1}.
	\end{split}	
\end{align}

Now we calculate the second term on the RHS of (\ref{2.10}). A direct consequence of (\ref{1.1}) is
\begin{align*}
   &\quad\sum_{i_1,\cdots, i_l=1}^r(\b{\p}\b{\p}^*_{i_1\cdots i_l}\mb{G}''_{i_1\cdots i_l}T_{\p\text{Res}_{D_{i_1}\cdots D_{i_l}}(\alpha)})(A_l\circ \cdots\circ A_1(\mathcal{I}^*(\beta)))\\
 =&(-1)^{p+q-l-1}\sum_{i_1,\cdots, i_l=1}^r(\b{\p}^*_{i_1\cdots i_l}\mb{G}''_{i_1\cdots i_l}T_{\p\text{Res}_{D_{i_1}\cdots D_{i_l}}(\alpha)})(\b{\p}A_l\circ \cdots\circ A_1(\mathcal{I}^*(\beta))).
\end{align*}
Then (\ref{1.7}), (\ref{2.1}), $\p\circ \mb{H}=0$ and also the K\"ahler identity induced from that on the smooth differential forms of K\"ahler manifolds imply

\begin{align*}
&(-1)^{p+q-l-1}\sum_{i_1,\cdots, i_l=1}^r(\b{\p}^*_{i_1\cdots i_l}\mb{G}''_{i_1\cdots i_l}T_{\p\text{Res}_{D_{i_1}\cdots D_{i_l}}(\alpha)})(\b{\p}A_l\circ \cdots\circ A_1(\mathcal{I}^*(\beta)))\\
	=&(-1)^{p+q-l-1}\sum_{i_1,\cdots, i_l=1}^r(\b{\p}^*_{i_1\cdots i_l}\mb{G}''_{i_1\cdots i_l}T_{\p\text{Res}_{D_{i_1}\cdots D_{i_l}}(\alpha)})(-A_l\circ\b{\p} A_{l-1}\circ \cdots\circ A_1(\mathcal{I}^*(\beta))\\
&\quad\quad\quad\quad\quad\quad\quad\quad\quad\quad\quad\quad\quad\quad\quad\quad\quad\quad\quad\quad+\iota^*_{i_1\cdots i_l}\circ A_{l-1}\circ \cdots\circ A_1(\mathcal{I}^*(\beta))).
\end{align*}
Since $i_{l-1}$ and $i_l$ anti-commute for $l\geq 2$ by (\ref{Reszero})  and Lemma \ref{lemma1} holds for $l=1$, one gets
\begin{align*}
&(-1)^{p+q-l-1}\sum_{i_1,\cdots, i_l=1}^r(\b{\p}^*_{i_1\cdots i_l}\mb{G}''_{i_1\cdots i_l}T_{\p\text{Res}_{D_{i_1}\cdots D_{i_l}}(\alpha)})(-A_l\circ\b{\p} A_{l-1}\circ \cdots\circ A_1(\mathcal{I}^*(\beta))\\
	&\quad\quad\quad\quad\quad\quad\quad\quad\quad\quad\quad\quad\quad\quad\quad\quad\quad\quad\quad\quad+\iota^*_{i_1\cdots i_l}\circ A_{l-1}\circ \cdots\circ A_1(\mathcal{I}^*(\beta)))\\
=&(-1)^{p+q-l}\sum_{i_1,\cdots, i_l=1}^r(\b{\p}^*_{i_1\cdots i_l}\mb{G}''_{i_1\cdots i_l}T_{\p\text{Res}_{D_{i_1}\cdots D_{i_l}}(\alpha)})(A_l\circ\b{\p}A_{l-1}\circ \cdots\circ A_1(\mathcal{I}^*(\beta))).
\end{align*}
This step is much inspired by Noguchi's trick on \cite[Page 298]{Nog}.
Repeat the three equalities above to obtain that
\begin{align*}
  &(-1)^{p+q-l}\sum_{i_1,\cdots, i_l=1}^r(\b{\p}^*_{i_1\cdots i_l}\mb{G}''_{i_1\cdots i_l}T_{\p\text{Res}_{D_{i_1}\cdots D_{i_l}}(\alpha)})(A_l\circ\b{\p}A_{l-1}\circ \cdots\circ A_1(\mathcal{I}^*(\beta)))\\
=&(-1)^{p+q-1}\sum_{i_1,\cdots, i_l=1}^r(\b{\p}^*_{i_1\cdots i_l}\mb{G}''_{i_1\cdots i_l}T_{\p\text{Res}_{D_{i_1}\cdots D_{i_l}}(\alpha)})(A_l\circ A_{l-1}\circ \cdots\circ A_1(\b{\p}\mathcal{I}^*(\beta))).
\end{align*}
In summary, the second term on the RHS of (\ref{2.10}) is
\begin{align}\label{2.12}
\begin{split}
	&\quad\sum_{i_1,\cdots, i_l=1}^r(\b{\p}\b{\p}^*_{i_1\cdots i_l}\mb{G}''_{i_1\cdots i_l}T_{\p\text{Res}_{D_{i_1}\cdots D_{i_l}}(\alpha)})(A_l\circ \cdots\circ A_1(\mathcal{I}^*(\beta)))\\
	=&(-1)^{p+q-1}\sum_{i_1,\cdots, i_l=1}^r(\b{\p}^*_{i_1\cdots i_l}\mb{G}''_{i_1\cdots i_l}T_{\p\text{Res}_{D_{i_1}\cdots D_{i_l}}(\alpha)})(A_l\circ A_{l-1}\circ \cdots\circ A_1(\b{\p}\mathcal{I}^*(\beta))).
\end{split}	
\end{align}

Let $$\chi_{i_1\cdots i_j}(T),\ j=0,\cdots,l,$$ be the current on $D_{i_1}\cap\cdots\cap D_{i_j}$ by trivial extension for any current $T$ on submanifold $S$ of  $D_{i_1}\cap\cdots\cap D_{i_j}$. More precisely, $$\chi_{i_1\cdots i_j}(T)(\beta)=T(\iota^*_{S}(\beta))$$ for any smooth form $\beta$ on $D_{i_1}\cap\cdots\cap D_{i_j}$. So the equation (\ref{2.12}) can be  reduced to
\begin{align}\label{2.14}
\begin{split}
&\quad\sum_{i_1,\cdots, i_l=1}^r(\b{\p}\b{\p}^*_{i_1\cdots i_l}\mb{G}''_{i_1\cdots i_l}T_{\p\text{Res}_{D_{i_1}\cdots D_{i_l}}(\alpha)})(A_l\circ \cdots\circ A_1(\mathcal{I}^*(\beta)))\\
&=\b{\p}\b{\p}^*\mb{G}''\Big((-1)^{\frac{l(l+1)}{2}}	\sum_{i_1,\cdots, i_l=1}^r(\chi_X\circ\b{\p}^*_{i_1}\mb{G}''_{i_1}\chi_{i_1}\circ\cdots\circ\b{\p}^*_{i_1\cdots i_{l-1}}\mb{G}''_{i_1\cdots i_{l-1}}\chi_{i_1\cdots i_{l-1}}\\
&\quad \quad\quad\quad\quad\quad \quad\quad\quad\quad \quad\quad\quad\circ\b{\p}^*_{i_1\cdots i_l}\mb{G}''_{i_1\cdots i_l}T_{\p\text{Res}_{D_{i_1}\cdots D_{i_l}}(\alpha)})\Big)(\mathcal{I}^*(\beta))\\
&\in \text{Im}(\p\b{\p}\b{\p}^*\mb{G}'')(\mathcal{I}^*(\beta)).
\end{split}
\end{align}

Substituting (\ref{2.13}) and (\ref{2.14}) into (\ref{2.10}), we have
$$
B_l\equiv 2\pi\sqrt{-1}(-1)^{p+q-(l+1)}B_{l+1},\quad \text{mod}\,\,\, \text{Im}(\p\b{\p}\b{\p}^*\mb{G}'')(\mathcal{I}^*(\beta)).	
$$
Then iteration and that $B_{n+1}=0$ imply
$$
B_l\in \text{Im}(\p\b{\p}\b{\p}^*\mb{G}'')(\mathcal{I}^*(\beta))
$$
for any $l\geq 1$.
\end{proof}

By considering the logarithmic $(p,q)$-forms as $E^p$-valued $(0,q)$-forms and using the bundle-valued Hodge Theorem (\ref{2.17}), one can solve a  $\b{\p}$-equation for logarithmic forms under an additional $\b\p$-exactness condition.
\begin{prop}\label{prop4}
Suppose that $\alpha\in A^{0,q}(X,\Omega^p_X(\log D))$ with $\b{\p}\alpha=0$ and $T_{\alpha}(\mathcal{I}^*(\beta))\in \text{Im}(\b{\p})(\mathcal{I}^*(\beta))$ for any $\beta\in A^{n,n-q}(X,(E^p)^*)$. Then there exists a $\gamma\in A^{0,q-1}(X,\Omega^p_X(\log D))$ such that
$$
\b{\p}\gamma=\alpha.	
$$
\end{prop}
\begin{proof}
Given an Hermitian metric $h$ on the vector bundle $E:=E^p$, let $\n=\n'+\b{\p}$ be the Chern connection of $(E,h)$. Then one has the  bundle-valued Hodge Theorem:
\begin{align}\label{2.17}
\mb{I}=\mb{H}''_E+\b{\p} \b{\p}_E^*\mb{G}''_E+ \b{\p}_E^*\b{\p}\mb{G}''_E,
\end{align}
 where $\mb{G}''_E$, $\mb{H}''_E$ and $ \b{\p}_E^*$ are Green's operator, harmonic projection operator and adjoint operator of $\b{\p}$ with respect to $h$ and $\omega$, respectively.

Applying (\ref{2.17}) to $\mathcal{I}(\alpha)\in A^{0,q}(X,E)$, one has
\begin{align}\label{2.20}
\begin{split}
	\mathcal{I}(\alpha)&=\mb{H}''_E(\mathcal{I}(\alpha))+\b{\p} \b{\p}_E^*\mb{G}''_E(\mathcal{I}(\alpha))+ \b{\p}_E^*\b{\p}\mb{G}''_E(\mathcal{I}(\alpha))\\
	&=\mb{H}''_E(\mathcal{I}(\alpha))+\b{\p} \b{\p}_E^*\mb{G}''_E(\mathcal{I}(\alpha)),
	\end{split}
\end{align}
where the last equality follows from $\b{\p}\circ \mathcal{I}=\mathcal{I}\circ \b{\p}$ and $\b{\p}\alpha=0$.

Let $e\in A^{0,q}(X,E)$ be a harmonic element, i.e., $\b{\p}e= \b{\p}_E^*e=0$.
Let $$*': A^{p,q}(X, E)\to A^{n-p,n-q}(X, E^*)$$ be the star operator, defined by
\begin{equation}\label{hodgestar}
 \eta(*'\theta)=\langle\eta,\theta\rangle\frac{\omega^n}{n!}	
\end{equation}
for any $\eta,\theta\in A^{p,q}(X, E)$, where $\langle\cdot,\cdot\rangle$ is the pointwise inner product induced by $(X,\omega)$ and $(E,h)$. So
$$
(\mathcal{I}(\alpha),e)=\int_X\langle\mathcal{I}(\alpha),e\rangle\frac{\omega^n}{n!}=\int_X \mathcal{I}(\alpha)(*'e)=T_{\alpha}(\mathcal{I}^* (*'e)),
$$
where the last equality follows from Lemma \ref{lemma3}. By the exactness of $T_{\alpha}$, there exists a current $T$ of bidegree $(p,q-1)$ on $X$ such that
$$T_{\alpha}(\mathcal{I}^* (*'e))=(\b{\p}T)(\mathcal{I}^* (*'e)).$$ So
\begin{align}\label{2.18}
	(\mathcal{I}(\alpha),e)=(-1)^{p+q}T(\b{\p}\mathcal{I}^* (*'e)).
\end{align}
Note that
\begin{align}\label{2.19}
\begin{split}
	\b{\p}\mathcal{I}^* (*'e)&=(-1)^{p}\mathcal{I}^*(\b{\p}*'e)\\
	&=(-1)^{p}\mathcal{I}^*((-1)^{q+1}*'*'\b{\p}*'e)\\
	&=(-1)^{p+q}\mathcal{I}^*(*' \b{\p}_E^*e)\\
    &=0,
\end{split}	
\end{align}
where the first equality follows from $\mathcal{I}^*\circ\b{\p}=(-1)^{p}\b{\p}\circ\mathcal{I}^*$ on $A^{n,n-q}(X,E^*)$. By (\ref{2.18}) and (\ref{2.19}), one has $\mb{H}''_E(\mathcal{I}(\alpha))=0$. Substituting it into
(\ref{2.20}), we have
\begin{align}\label{2.21}
\mathcal{I}(\alpha)=\b{\p} \b{\p}_E^*\mb{G}''_E(\mathcal{I}(\alpha)).
\end{align}
Since $\mathcal{I}: A^{0,q}(X,\Omega^p_X(\log D))\to A^{0,q}(X, E)$ is an isomorphism,
\begin{align}\label{gamma}
\gamma:=\mathcal{I}^{-1}\left( \b{\p}_E^*\mb{G}''_E(\mathcal{I}(\alpha))\right)\in A^{0,q-1}(X,\Omega^p_X(\log D)).
\end{align}
Applying $\mathcal{I}^{-1}$ to both sides of (\ref{2.21}) and using $\b{\p}\circ\mathcal{I}^{-1}=\mathcal{I}^{-1}\circ\b{\p}$, we have
$$
 \b{\p}\gamma=\alpha.	
$$
\end{proof}

Using Lemma \ref{lemma2} and Proposition \ref{prop4}, we get the first main theorem.
\begin{thm}\label{thm2}
For any $\alpha\in A^{0,q}(X,\Omega^p_X(\log D))$ with $\b{\p}\p\alpha=0$, there exists a solution $x\in A^{0,q-1}(X,\Omega^{p+1}_X(\log D))$ such that
\begin{align}\label{2.5}
\b{\p}x=\p\alpha.
\end{align}
\end{thm}
\begin{proof}
From (\ref{1.7}), (\ref{2.1}) and (\ref{2.6}), it follows that
\begin{align}\label{2.7}
\begin{split}
T_{\p\alpha}&=\b{\p}\b{\p}^*\mb{G}''T_{\p\alpha}-2\pi\sqrt{-1}\b{\p}^*\mb{G}''\p\text{Res}(\alpha).
\end{split}
\end{align}

By (\ref{2.2}), for any $\beta\in A^{n,n-q}(X, (E^{p+1})^*)$, one has
\begin{align}\label{2.8}
\begin{split}
	\b{\p}^*\mb{G}''\p\text{Res}(\alpha)(\mathcal{I}^*(\beta))&=\text{Res}(\alpha)(\p\mb{G}''\b{\p}^*\mathcal{I}^*(\beta))\\
	&=\sum_{i=1}^r\int_{D_i}\text{Res}_{D_i}(\alpha)\wedge \iota_{D_i}^*(\p\mb{G}''\b{\p}^*\mathcal{I}^*(\beta))\\
	&=\sum_{i=1}^r\int_{D_i}(-1)^{p+q}\p\text{Res}_{D_i}(\alpha)\wedge \iota_{D_i}^*(\mb{G}''\b{\p}^*\mathcal{I}^*(\beta))\\
	&=(-1)^{p+q}\sum_{i=1}^r T_{\p\text{Res}_{D_i}(\alpha)}(\iota_{D_i}^*\b{\p}^*\mb{G}''\mathcal{I}^*(\beta)).
	\end{split}
\end{align}
By Lemma \ref{lemma2}, one gets
$$
	\b{\p}^*\mb{G}''\p\text{Res}(\alpha)(\mathcal{I}^*(\beta))\in \text{Im}(\p\b{\p}\b{\p}^*\mb{G}'')(\mathcal{I}^*(\beta)).
$$
Therefore,
$$
T_{\p\alpha}(\mathcal{I}^*(\beta))=(\b{\p}\b{\p}^*\mb{G}''T_{\p\alpha}-2\pi\sqrt{-1}\b{\p}^*\mb{G}''\p\text{Res}(\alpha))(\mathcal{I}^*(\beta))\in \text{Im}(\p\b{\p}\b{\p}^*\mb{G}'')(\mathcal{I}^*(\beta)).
$$
By Proposition \ref{prop4}, there exists a solution in $A^{0,q-1}(X,\Omega^{p+1}_X(\log D))$ for the equation (\ref{2.5}).
\end{proof}

Using an analogous argument for Theorem \ref{thm2}, one can solve another $\b{\p}$-equation:
\begin{thm}\label{sol-CY}
If $\alpha\in A^{n,n-q}(X, T_X^p(-\log D))\subset A^{n-p,n-q}(X)$ with $\b{\p}\p\alpha=0$, then there is a solution $x\in A^{n,n-q-1}(X, T_X^{p-1}(-\log D))$ such that
$$
\b{\p}x=\p\alpha.	
$$
\end{thm}
\begin{proof}
Let $(E^p)^*$ be the hermitian holomorphic vector bundle associated with the locally free sheaf $T_X^p(-\log D)$, and
$$\xymatrix{
&A^{n-p,n-q}(X)&\\
(\mathcal{I}^*)^{-1}:&A^{n,n-q}(X, T_X^p(-\log D))\ar@{^{(}->}[u]\ar[r]&A^{n,n-q}(X, (E^p)^*)
}$$
the canonical isomorphism with $\mathcal{I}^*$ given by \eqref{iso2}. Notice that
$$
\p\alpha\in A^{n,n-q}(X, T^{p-1}_X(-\log D))\subset A^{n-p+1,n-q}(X).	
$$
In fact, for any $z\in U\subset X$, we may assume that $D\cap U=\{z^1\cdots z^k=0\} $ around $z$ and locally,
$$
\alpha=f_{i_{1}\cdots i_{p}}(z^{i_{1}})^{\sigma(i_{1})}\cdots (z^{i_{p}})^{\sigma(i_{p})}i_{\frac{\p\ }{\p z^{i_{1}}}}\circ\cdots\circ i_{\frac{\p\ }{\p z^{i_p}}}(dz^1\wedge \cdots\wedge dz^n),
$$
where $f_{i_{1}\cdots i_{p}}$ is a locally defined smooth $(0,n-q)$-form and the function $\sigma(\cdot)$ is defined by
\[ \sigma(i_j)=\begin{cases}
1,\ &i_j\in \{1,\cdots, k\}, \\[4pt]
0,\ &\text{otherwise}.
\end{cases} \]
Thus,
\begin{align*}
\p\alpha&=\sum_{l=1}^{p}\frac{\p (f_{i_{1}\cdots i_{p}}(z^{i_{1}})^{\sigma(i_{1})}\cdots (z^{i_{p}})^{\sigma(i_{p})})}{\p z^{i_l}}(-1)^{n-q+l-1} i_{\frac{\p\ }{\p z^{i_{1}}}}\circ\cdots\widehat{i_{\frac{\p\ }{\p z^{i_{l}}}}}\cdots\circ i_{\frac{\p\ }{\p z^{i_p}}}(dz^1\wedge \cdots\wedge dz^n)\\
&=\sum_{l=1}^{p}(-1)^{n-q+l-1}\left(\frac{\p f_{i_{1}\cdots i_{p}}}{\p z^{i_l}}(z^{i_{l}})^{\sigma(i_l)}+\sigma(i_l)f_{i_{1}\cdots i_{p}}\right)\\
&\quad (z^{i_{1}})^{\sigma(i_{1})}\cdots \widehat{(z^{i_l})^{\sigma(i_l)}}\cdots(z^{i_{p}})^{\sigma(i_{p})}i_{\frac{\p\ }{\p z^{i_{1}}}}\circ\cdots\widehat{i_{\frac{\p\ }{\p z^{i_{l}}}}}\cdots\circ i_{\frac{\p\ }{\p z^{i_p}}}(dz^1\wedge \cdots\wedge dz^n),
\end{align*}
which lies in $A^{n,n-q}(X, T^{p-1}_X(-\log D))\subset A^{n-p+1,n-q}(X)$ since the coefficient
$$
	\frac{\p f_{i_{1}\cdots i_{p}}}{\p z^{i_l}}(z^{i_{l}})^{\sigma(i_l)}+\sigma(i_l)f_{i_{1}\cdots i_{p}}
$$
is smooth.

Set $E=(E^{p-1})^*$ and $h$ as a hermitian metric on it. By bundle-valued Hodge Theorem (\ref{2.17}), one has
\begin{align}\label{log 4.1}
\begin{split}
(\mathcal{I}^*)^{-1}(\p\alpha)&=\mb{H}''_E((\mathcal{I}^*)^{-1}(\p\alpha))+\b{\p} \b{\p}_E^*\mb{G}''_E(\mathcal{I}^*)^{-1}(\p\alpha)\\
&=\mb{H}''_E((\mathcal{I}^*)^{-1}(\b{\p}\p\beta))+\b{\p} \b{\p}_E^*\mb{G}''_E(\mathcal{I}^*)^{-1}(\p\alpha).
\end{split}
\end{align}
Here the smooth complex $(n-p,n-q-1)$-form $\beta$ is chosen as
$$-\b{\p}^*\mb{G}''\alpha$$
according to Hodge decomposition theorem, K\"ahler identity on $X$ and the assumption $\b{\p}\p\alpha=0$.

Now we claim
\begin{equation}\label{h-vanishing}
  \mb{H}''_E((\mathcal{I}^*)^{-1}(\b{\p}\p\beta))=0.
\end{equation}
In fact,
let $e\in A^{0,q}(X,E)$ be a harmonic element, i.e., $\b{\p}e= \b{\p}_E^*e=0$, and
$$*': A^{p,q}(X, E)\to A^{n-p,n-q}(X, E^*)$$ the star operator, defined similar by \eqref{hodgestar}. So
\begin{align*}
((\mathcal{I}^*)^{-1}(\p\alpha),e)
 &=\int_X\langle(\mathcal{I}^*)^{-1}(\p\alpha),e\rangle\frac{\omega^n}{n!}\\
 &=\int_X (\mathcal{I}^*)^{-1}(\p\alpha)\wedge *'e\\
 &=\int_X \p\alpha\wedge \mathcal{I}^{-1}(*'e),
\end{align*}
where $\omega$ is a K\"ahler form on $X$ and the last equality follows from the analogous Lemma \ref{lemma3}.
Then
\begin{align*}
\int_X \p\alpha\wedge \mathcal{I}^{-1}(*'e)
 &=\int_X \b\p\p\beta\wedge \mathcal{I}^{-1}(*'e)\\
 &=\int_X (\b\p(\p\beta\wedge \mathcal{I}^{-1}(*'e))+(-1)^{p+q+1}\p\beta\wedge \b\p(\mathcal{I}^{-1}(*'e)))\\
 &=\int_X \b\p(\p\beta\wedge \mathcal{I}^{-1}(*'e)),
\end{align*}
where the last equality follows from an analogous vanishing argument of \eqref{2.19}. Notice that $\mathcal{I}^{-1}(*'e)$
is a  logarithmic $(p-1,q)$-form. By the formula \eqref{2.6}, one gets
\begin{align*}
\int_X \b\p(\p\beta\wedge \mathcal{I}^{-1}(*'e))
 &=\pm2\pi\sqrt{-1}\text{Res}(\mathcal{I}^{-1}(*'e))(\p\beta)\\
 &=\pm2\pi\sqrt{-1}\text{Res}(\mathcal{I}^{-1}(*'e))(\p\b{\p}^*\mb{G}''\alpha)\\
 &=\pm2\pi\sqrt{-1}\p\b{\p}^*\mb{G}''\text{Res}(\mathcal{I}^{-1}(*'e))(\alpha).
\end{align*}
As reasoned in \eqref{2.8}, Lemma \ref{lemma2} implies the existence of some current $C$ on $X$ such that
\begin{align*}
\p\b{\p}^*\mb{G}''\text{Res}(\mathcal{I}^{-1}(*'e))(\alpha)
 &=\int_X \p\b{\p} C\wedge \alpha\\
 &=\int_X C\wedge\b{\p}\p \alpha\\
 &=0,
 \end{align*}
where the last equality is got by the assumption $\b{\p}\p\alpha=0$. So we have proved the claim.

  Substituting \eqref{h-vanishing} into (\ref{log 4.1}), we have
$$
 	(\mathcal{I}^*)^{-1}(\p\alpha)=\b{\p} \b{\p}_E^*\mb{G}''_E(\mathcal{I}^*)^{-1}(\p\alpha).
$$
Therefore, one can find a solution
$$
x=\mathcal{I}^* \b{\p}_E^*\mb{G}''_E(\mathcal{I}^*)^{-1}(\p\alpha)\in A^{n,n-q-1}(X, T_X^{p-1}(-\log D)).
$$
\end{proof}

\section{Applications to algebraic geometry} \label{section3}

In this section, we will give three kinds of applications of Theorems $\ref{thm2}$ and \ref{sol-CY} to algebraic geometry. Throughout this section,
let $D$ be a simple normal crossing divisor on a compact K\"ahler manifold $X$.

\subsection{Closedness of logarithmic forms}

The theory of logarithmic forms has been playing very important roles in various aspects of analytic-algebraic geometry, in which the understanding of closedness of logarithmic forms is fundamental. In $1971$, Deligne \cite[(3.2.14)]{Del71} proved the $d$-closedness of logarithmic forms on a smooth complex quasi-projective variety by showing the degeneracy of logarithmic Hodge to de Rham spectral sequence, which is also to be studied in the next subsection. In fact, his proof works for a Zariski open subspace of a compact K\"ahler manifold. In $1995$, by using classical harmonic integral theory \cite{de,Kodaira} by de Rham and Kodaira, Noguchi \cite{Nog} gave a short proof of this result.

As the first application of Theorem $\ref{thm2}$, we generalize Deligne's result on the closeness of logarithmic forms \cite{Del71}, compared with the one in \cite{Nog}.
\begin{cor}\label{g-No}
	For any $\alpha\in A^{0,0}(X,\Omega^p_X(\log D))$ with $\b{\p}\p\alpha=0$, $\p\alpha=0$.
\end{cor}
\begin{proof}
	By Theorem \ref{thm2}, there exists a solution $$x\in A^{0,-1}(X,\Omega^p_X(\log D))=\{0\}$$ such that $\b{\p}x=\p\alpha$ and then
	$$
	\p\alpha=\b{\p}x=0
	$$
since $x=0$.
\end{proof}

\subsection{Degeneracy of spectral sequences}
As the second application of Theorems $\ref{thm2}$ and \ref{sol-CY}, we reprove Deligne's degeneracy of the logarithmic Hodge to de Rham spectral sequence at $E_1$ \cite{Del71} and also
its dual version on a compact K\"ahler manifold, respectively.

Two nice references for spectral sequences should be \cite{Griffith,V}.
Let $(X,D)$ be as above and $U=X-D$. One can show
\begin{equation}\label{drhc}
  H^k(U,\mathbb{C})=\mb{H}^{k}(X, \Omega^*_X(\log D))
\end{equation}
as \cite[Corollary 8.19]{V} or \cite[p. 453]{Griffith}.
The complex $\Omega^*_X(\log D)$ is equipped with the "naive" filtration, which induces
a filtration on $H^k(U,\mathbb{C})$, called the \emph{Hodge filtration} of $H^k(U,\mathbb{C})$:
$$F^pH^k(U,\mathbb{C})=\Im(\mb{H}^{k}(X, \Omega^{\geq p}_X(\log D))\rightarrow\mb{H}^{k}(X, \Omega^*_X(\log D))).$$
As for the holomorphic de Rham complex, the spectral sequence associated to the Hodge filtration on $\Omega^*_X(\log D)$
has first term equal to $$E^{p,q}_1=H^q(X,\Omega^p_X(\log D)),$$ where the differential is induced by $\partial$.

\begin{thm}\label{Dss}
	The spectral sequence associated with the Hodge filtration
	$$
	E^{p,q}_1=H^q(X,\Omega^p_X(\log D))\Rightarrow \mb{H}^{p+q}(X, \Omega^*_X(\log D))	
	$$
degenerates at the $E_1$-level.
\end{thm}
\begin{proof}
The proof needs a logarithmic analogue of the general description on the terms in the Fr\"olicher spectral sequence as in \cite[Theorems 1 and 3]{cfgu}.
By Dolbeault isomorphism theorem, one has
$$
H^q(X,\Omega^p_X(\log D))\cong H^{0,q}_{\b{\p}}(X,\Omega^p_X(\log D)):=\frac{\text{Ker}(\b{\p})\cap A^{0,q}(X,\Omega^p_X(\log D))}{\b{\p}A^{0,q-1}(X,\Omega^p_X(\log D))}.
$$
Now we want to interpret $$E_r^{p,q}\cong Z_r^{p,q}/B_r^{p,q},$$
where $Z_r^{p,q}$ lies between the $\b{\p}$-closed and $d$-closed logarithmic $(p,q)$-forms and $B_r^{p,q}$ lies between the $\b{\p}$-exact and $d$-exact logarithmic $(p,q)$-forms
in some senses.
Actually,
$$Z_1^{p,q}=\{\alpha\in A^{0,q}(X,\Omega^p_X(\log D))\ |\ \b{\p}\alpha=0\},$$
$$B_1^{p,q}=\{\alpha\in A^{0,q}(X,\Omega^p_X(\log D))\ |\ \alpha=\b{\p}\beta, \beta\in A^{0,q-1}(X,\Omega^p_X(\log D))\}.$$
For $r\geq 2$,
\begin{equation}\label{zr}
\begin{aligned}
Z_r^{p,q}=\{\alpha_{p,q}\in A^{0,q}(X,\Omega^p_X(\log D))\ |
 &\ \b{\p}\alpha_{p,q}=0,\ \text{and there exist }\\
 &\text{$\alpha_{p+i,q-i}\in A^{0,q-i}(X,\Omega^{p+i}_X(\log D))$ }\\
  &\text{such that $\p\alpha_{p+i-1,q-i+1}+\b\p\alpha_{p+i,q-i}=0, 1\leq i\leq r-1$}
 \},
\end{aligned}
\end{equation}
\begin{align*}
B_r^{p,q}=\{&\p\beta_{p-1,q}+\b\p\beta_{p,q-1}\in A^{0,q}(X,\Omega^p_X(\log D))\ |
 \ \text{there exist }\\
 &\text{ $\beta_{p-i,q+i-1}\in A^{0,q+i-1}(X,\Omega^{p-i}_X(\log D))$, $2\leq i\leq r-1, $}\\
 &\text{ such that $\p\beta_{p-i,q+i-1}+\b\p\beta_{p-i+1,q+i-2}=0, \b\p\beta_{p-r+1,q+r-2}=0$}\},
\end{align*}
and the map $d_r: E_r^{p,q}\longrightarrow E_r^{p+r,q-r+1}$ is given by
$$d_r[\alpha_{p,q}]=[\p\alpha_{p+r-1,q-r+1}],$$
where $[\alpha_{p,q}]\in E_r^{p,q}$ and $\alpha_{p+r-1,q-r+1}$ appears in \eqref{zr}. Notice that $\b\p\alpha_{p+r-1,q-r+1}$
doesn't necessarily vanish for $r\geq 2$. Hence, a direct and exact application of Theorem \ref{thm2} implies
$$d_i=0,\ \forall i\geq 1,$$
which is indeed the desired degeneracy.
\end{proof}

From the Theorem \ref{Dss} and \eqref{drhc}, one has a non-canonical logarithmic Hodge decomposition.
\begin{cor}\label{decomposition}
Let $X$ be a compact K\"ahler manifold and $D$ a simple normal crossing divisor on $X$. Then
$$
\dim_\mb{C}H^k(X-D,\mb{C})=\sum_{p+q=k}\dim_\mb{C} H^q(X,\Omega^p_X(\log D)).
$$
\end{cor}

Similar to Theorem \ref{Dss}, one can use Theorem \ref{sol-CY} to prove a dual version of Theorem \ref{Dss} on compact K\"ahler manifolds.
\begin{cor}\label{Ddual}
The spectral sequence associated with the Hodge filtration
\begin{align*}
E^{p,q}_1=H^q(X,\Omega^p_X(\log D)\otimes \mc{O}_X(-D))\Rightarrow 	\mb{H}^{p+q}(X,\Omega^*_X(\log D)\otimes \mc{O}_X(-D))
\end{align*}
	degenerates at $E_1$-level.
\end{cor}
\begin{proof}
By the  isomorphism
$$
(\Omega^p_X(\log D))^*\cong \Omega^{n-p}_X(\log D)\otimes \mc{O}_X(-K_X-D),	
$$
we have
$$
 A^{0,n-q}(X,\Omega^{n-p}_X(\log D)\otimes \mc{O}_X(-D))=A^{n,n-q}(X, T_X^p(-\log D))\subset A^{n-p,n-q}(X).	
$$
By Theorem \ref{sol-CY}, for any $\alpha\in  A^{0,n-q}(X,\Omega^{n-p}_X(\log D)\otimes \mc{O}_X(-D))$ with $\b{\p}\p\alpha=0$, there exists $x\in  A^{0,n-q-1}(X,\Omega^{n-p+1}_X(\log D)\otimes \mc{O}_X(-D))$ such that
$$\b{\p}x=\p\alpha.$$
From the same argument as in Theorem \ref{Dss}, it follows that the spectral sequence
	\begin{align*}
E^{p,q}_1=H^q(X,\Omega^p_X(\log D)\otimes \mc{O}_X(-D))\Rightarrow 	\mb{H}^{p+q}(X,\Omega^*_X(\log D)\otimes \mc{O}_X(-D))
\end{align*}
	degenerates at $E_1$-level.
\end{proof}
\begin{rem}
If $X$ is a proper smooth algebraic variety over $\mb{C}$, a proof of the above result was given by O. Fujino in \cite[Section 2.29]{Fujino}.
The duality between Theorem \ref{Dss} and Corollary \ref{Ddual} was pointed out in \cite[Remark 2.11]{Fujinosjv}.
\end{rem}

\subsection{Injectivity theorem}
The third result of Theorem \ref{main theorem} is an injectivity theorem for compact K\"ahler manifolds, whose algebraic version was first proved by F. Ambro \cite[Theorem 2.1]{Ambro}.
\begin{cor}\label{inj}
Let $X$ be a compact K\"ahler manifold and $D$ a simple normal crossing divisor. Then the restriction homomorphism
$$
	H^q(X,\Omega^n_X(\log D))\xrightarrow{i} H^q(U, K_U)
$$
	is injective, where $U=X-D$. Equivalently, if $\Delta$ is an effective divisor with $\text{Supp}(\Delta)\subset \text{Supp}(D)$, then
 the natural homomorphism induced by the inclusion $\mc{O}_X\subset \mc{O}_X(\Delta)$
$$
		H^q(X,\Omega^n_X(\log D))\xrightarrow{i'} H^q(X, \Omega^n_X(\log D)\otimes\mc{O}_X(\Delta))
$$
is injective.
\end{cor}
\begin{proof}
Consider the commutative diagram:
\begin{equation}\label{comute}
\begin{CD}
H^q(X,\Omega^n_X(\log D)) @>{i}>> H^{q}(U,K_U)\\
@VV{j_1}V @ VV {j_2} V\\
H^{q+n}(X,\Omega_{X}^{\bullet}(\log D)) @>{\cong}>> H^{q+n}(U,\mb{C}),
\end{CD}
\end{equation}
where both $j_1$, $j_2$ are induced by the identity maps, and the isomorphism $H^{q+n}(X,\Omega^{\cdot}_X(\log D))\cong H^{q+n}(U,\mb{C})$ appears as \cite[Theorem 1.3 in Part III]{Ancona}.
For any element $[\alpha]\in H^q(X,\Omega^n_X(\log D))$,
$$
j_1([\alpha])=[\alpha]_d\in H^{q+n}(X,\Omega^{\bullet}_X(\log D)).	
$$
If $j_1([\alpha])=0$, i.e., $[\alpha]_d=0$, there exists a logarithmic form $\beta\in A^{q+n-1}(X,\Omega^{\bullet}_X(\log D))$ such that
$$
	\alpha=d\beta.
$$
Therefore, the components
 $\beta_{n-1,q}\in A^{0,q}(X,\Omega^{n-1}_X(\log D))$ and $\beta_{n,q-1}\in A^{0,q-1}(X,\Omega^{n}(\log D))$ of $\beta$ satisfy
 \begin{align}\label{injective1}
\alpha=\p\beta_{n-1,q}+\b{\p}\beta_{n,q-1}.
 \end{align}
Notice that $\b{\p}(\p\beta_{n-1,q})=\b{\p}\alpha=0$, by Theorem \ref{thm2}, and  there exists $\gamma\in A^{0,q-1}(X,\Omega^{n}(\log D))$ such that
\begin{align}\label{injective2}
\p\beta_{n-1,q}=\b{\p}\gamma.
\end{align}
Combining (\ref{injective1}) with (\ref{injective2}), one has
$$
\alpha=\b{\p}(\gamma+\beta_{n,q-1}),	
$$
which implies that $[\alpha]=0\in H^q(X,\Omega^n_X(\log D))$. So we get the injectivity of $j_1$ and thus that of $i$ by the commutative diagram \eqref{comute}.

 Now we prove the injectivity of $i'$. The mapping $i$ can be decomposed into
$$\xymatrix{&H^q(X, \Omega^n_X(\log D)\otimes\mc{O}_X(\Delta))\ar[dr]^{j'}&\\
H^q(X,\Omega^n_X(\log D))\ar[ur]^{i'}\ar[rr]_{i} & & H^q(U,K_U),
}$$
where $j'$ is induced by the identity map. So the injectivity of $i$ implies that of $i'$.
We remark that the equivalence here was first proposed in \cite[Remark 2.6 and Corollary 2.7]{Ambro}
\end{proof}
\begin{rem}
O. Fujino  \cite[Theorem 1.1]{Fujinoinj} generalized Ambro's algebraic version to a simple normal crossing algebraic variety.
\end{rem}

\section{Applications to logarithmic deformations} \label{section4}

We will present three applications of Theorems $\ref{thm2}$ and \ref{sol-CY} to logarithmic deformations in this section. Throughout this section, we assume that
 $X$ is a compact K\"ahler manifold and $D$ is a simple normal crossing divisor on $X$.

Firstly, we recall basic notions and properties of logarithmic deformations in \cite{kawa}. 
 \begin{defn}[\cite{kawa}]\label{kwawa1}
 \textit{A family of logarithmic deformations} of a pair $(X,D)$ is a $7$-tuple $\mathscr{F}=(\mathscr{X},\b{\mathscr{X}},\b{\mathscr{D}}, \b{\pi},S,s_0,\b{\psi})$ satisfying the following conditions:
  \begin{enumerate}
    \item $\b{\pi}:\b{\mathscr{X}}\to S$ is a proper smooth morphism of complex space $\b{\mathscr{X}}$ and $S$.
    \item $\b{\mathscr{D}}$ is  a closed analytic subset of $\b{\mathscr{X}}$ and $\mathscr{X}=\b{\mathscr{X}}-\b{\mathscr{D}}$.
    \item $\b{\psi}:X\to \b{\pi}^{-1}(s_0)$ is an isomorphism such that $\b{\psi}(X-D)=\b{\pi}^{-1}(s_0)\cap \mathscr{X}$.
    \item $\b{\pi}$ is  locally a projection of a product space as well as the restriction of it to $\b{\mathscr{D}}$, that is, for each $p\in \b{\mathscr{X}}$ there exist an open neighborhood $U$ of $p$ and an isomorphism $\mu:U\to V\times W$, where $V=\b{\pi}(U)$ and $W=U\cap\b{\pi}^{-1}(\b{\pi}(p))$, such that the following diagram
$$\xymatrix{U\ar[dr]_{\b{\pi}}\ar[rr]^{\mu} & & V\times W\ar[dl]^{\text{pr}_1}\\
&V&}$$
  commutes and $\mu(U\cap\b{\mathscr{D}})=V\times(W\cap\b{\mathscr{D}})$.
  \end{enumerate}
\end{defn}

Moreover, a family $\mathscr{F}=(\mathscr{X},\b{\mathscr{X}},\b{\mathscr{D}}, \b{\pi},S,s_0,\b{\psi})$ of pair $(X,D)$ is called \textit{semi-universal} (cf. \cite[Definition 5]{kawa}) if for any family $\mathscr{F}'=(\mathscr{X}',\b{\mathscr{X}}',\b{\mathscr{D}}', \b{\pi'},S',s_0',\b{\psi}')$ of logarithmic deformations of $(X,D)$ there exist an open neighborhood $S''$ of $s'_0$ in $S'$ and a morphism $\alpha:S''\to S$ such that
\begin{itemize}
\item[1)] The restriction $\mathscr{F}'|_{S''}$ of $\mathscr{F}'$ over $S''$ is isomorphic to the induced family $\alpha^*\mathscr{F}$,
\item[2)] For any $S_0''$ and $\alpha_0$ satisfying the same condition as in $1)$, the induced tangential maps $T_{\alpha}$ and $T_{\alpha_0}$ from $T_{S',s_0'}$ to $T_{S,s_0}$ coincide.
\end{itemize}

In \cite[Theorem 1]{kawa}, Y. Kawamata proved the following Kuranishi type theorem, whose proof also plays an important role in constructing extension of logarithmic forms.
\begin{thm}[\cite{kawa}]\label{thm3.1}
There exists a semi-universal family $\mathscr{F}$ of logarithmic deformations of the pair $(X,D)$.
\end{thm}

Let $T_X(-\log D)$ be the dual sheaf of $\Omega^1_X(\log D)$. Then the set of infinitesimal logarithmic deformations is $H^1(X, T_X(-\log D))$  (cf. \cite{Groth}). Moreover, as shown on \cite[p. 251]{kawa}, the semi-universal family $\mathscr{F}$ in Theorem \ref{thm3.1} can be obtained from a subspace of
$$\Gamma_{\text{real analytic}}(X, T_X(-\log D)\otimes \Lambda^{0,1}T^*X),$$ which, usually called the space of \emph{Beltrami differentials},  consists of sections  satisfying the integrability  condition:
\begin{align}\label{integrable}
\b{\p}\varphi=\frac{1}{2}[\varphi,\varphi].
\end{align}

\subsection{Extension of logarithmic forms}

In this subsection, we consider the extension problem of $\b{\p}$-closed logarithmic $(n,q)$-form under the logarithmic deformations on a K\"ahler manifold and obtain the local stabilities of log Calabi-Yau structures. For the case of smooth $(n,q)$-form, a good reference is \cite{Tujie}.

Let $\mathscr{F}=(X_t, D_t)$, $t\in S$ be a family of logarithmic deformations of pair $(X,D)$. For any $\b{\p}$-closed $(n,q)$-logarithmic form $\Omega\in A^{0,q}(X,\Omega^n_X(\log D))$, we want to extend this form $\Omega$ to $\cup_{t\in\Delta}X_t$ smoothly for some small neighborhood $\Delta\subset S$ of the reference point $s_0\in S$, and thus get a $\b{\p}$-closed logarithmic $(n,q)$-form when restricted to each $X_t$, $t\in \Delta$.

Without loss of generality, we may assume that $\mathscr{F}=(X_t, D_t)$, $t\in S$ is a semi-universal family. In fact, if we assume $\Omega_t$, $t\in \Delta\subset S$ is a smooth extension of $\Omega$ on the semi-universal family $\mathscr{F}|_{\Delta}=(X_t, D_t), t\in \Delta$, and assume that $\mathscr{F}'=(X'_t, D'_t)$, $t\in S'$ is another family of logarithmic deformation of pair $(X,D)$. By the definition of semi-universal family,  there exist an open neighborhood $\Delta'$ of $s_0'$ and a morphism $\alpha:\Delta'\to S$, and thus $\alpha^*\Omega_t$ gives a $\b{\p}$-closed extension of $\Omega$ on the family $\mc{F}'|_{\alpha^{-1}(\Delta\cap\alpha(\Delta'))}=(X'_t,D'_t)$, $t\in \alpha^{-1}(\Delta\cap\alpha(\Delta'))$. Similar reductions to this can be found in \cite[Subsection 2.3]{RZ15} and the beginning of \cite[Section 2]{RwZ}.

Let
$$
	\varphi:=\varphi(t)\in \Gamma_{\text{real analytic}}(X, T_X(-\log D)\otimes \Lambda^{0,1}T^*X)\subset A^{0,1}(X, T_X(-\log D))
$$
 be the Beltrami differential, which satisfies  the integrability condition (\ref{integrable}) and  gives the semi-universal family $\mathscr{F}=(X_t, D_t)$, $t\in S$.

By a direct calculation, the contraction map
$$\label{varphi}
  \varphi\l :=i_{\varphi}: A^{0,q}(X,\Omega^p_X(\log D))\to A^{0,q+1}(X, \Omega^{p-1}_X(\log D))
$$
is well-defined.
 As in \cite{Liu,RZ15}, one may define the operator
$$
	e^{i_\varphi}:=\sum_{k=0}^{\infty}\frac{1}{k!}i^k_{\varphi},
$$
where $i^k_{\varphi}=\underbrace{i_{\varphi}\circ\cdots\circ i_{\varphi}}_k$. Notice that the above summation should be finite
due to the dimension assumption.
\begin{prop}\label{prop2} With the above setting, the operator
$$
		e^{i_\varphi}: A^{0,q}(X,\Omega^n_X(\log D))\to A^{0,q}(X_t,\Omega^n_{X_t}(\log D_t))
$$
	is a linear isomorphism as $|t-s_0|$ is small.
\end{prop}
\begin{proof}
Since $e^{i_\varphi}: A^{n,q}(X)\to A^{n,q}(X_t)$ is a linear isomorphism  as $|t-s_0|$ is small, it suffices to prove that $e^{i_\varphi}(\alpha)$ has logarithmic poles along $D_t$ for any $\alpha\in A^{0,q}(X,\Omega^n_X(\log D))$.

Let $\underline{X}$ and $\underline{D}$ be the (union of) the underlying real analytic manifolds of $X$ and $D$.
Choose any point $p$ in $\underline X$. Set $\zeta=\zeta(z,t)$ as a local holomorphic coordinate system of $X_t$ around $p$ induced by the family
and
$$R_t(\zeta(z,t)), t\in S$$
 as a local defining function of $D_t$ around $p$. By (the proof of) \cite[Lemma 1]{kawa},
since $\varphi$ is logarithmic along $D$ and thus $\underline{D}\times S$ is analytic with respect to the complex structure
given by $\varphi$, there exists a real analytic function
$$\epsilon_{t}(z)$$
 around $p$ parameterized by $S$ such that $\epsilon_{s_0}(z)\neq 0$ and
 $$R_{s_0}(\zeta(z,{s_0}))\epsilon_t(z)$$ is a local holomorphic function with respect to the complex structure
given by $\varphi$ on $X_t$. As $R_{s_0}(\zeta(z,{s_0}))\epsilon_{t}(z)$ vanishes on the union  $\underline{D}$ of underlying real analytic manifolds,
there exists some local holomorphic function $h(\zeta,t)$ on $X_t$ such that
\begin{align}\label{2.23}
	R_{s_0}(\zeta(z,{s_0}))\epsilon_{t}(z)=h(\zeta,t)R_t(\zeta(z,t)). 	
\end{align}
Since $h(\zeta,s_0)=\epsilon_{s_0}(z)\neq 0$ on a (possibly smaller) neighborhood of $p$ by \cite[Theorem 6.6 of Chapter II]{Dem}, $h(\zeta,t)$ has no zero points around $p$  as $|t-s_0|$ is small. From  (\ref{2.23}), one has
\begin{align}\label{2.24}
	\frac{R_t(\zeta(z,t))}{R_{s_0}(\zeta(z,{s_0}))}=\frac{\epsilon_{t}(z)}{h(\zeta,t)},
\end{align}
which is smooth and has no zero points around $p$ and $s_0$.

Now any $\alpha\in A^{0,q}(X,\Omega^n_X(\log D))$ is locally written as
$$\alpha=\frac{\alpha_1}{R_{s_0}(\zeta(z,{s_0}))}$$
with some local smooth $(n,q)$-form $\alpha_1$ on $X$.
So
\begin{align*}
e^{i_\varphi}(\alpha)&=\frac{1}{R_{s_0}(\zeta(z,{s_0}))}e^{i_\varphi}( \alpha_1)=\frac{R_t(\zeta(z,t))}{R_{s_0}(\zeta(z,{s_0}))}\left(\frac{1}{R_t(\zeta(z,t))}e^{i_\varphi}(\alpha_1)\right)
\end{align*}
lies in $A^{0,q}(X_t, \Omega^n_{X_t}(\log D_t))$ by (\ref{2.24}) and that $e^{i_\varphi}(\alpha_1)$ is a local smooth $(n,q)$-form on $X_t$.
\end{proof}

Without loss of generality, we may assume that $S=\Delta$ is a small disc and $s_0=0$.
Let $\Omega\in A^{0,q}(X,\Omega^n(\log D))$ be a $\b{\p}$-closed logarithmic $(n,q)$-form on $X$.
In order to find a smooth $\b{\p}$-closed extension of $\Omega$, we only need to find a real analytic $\Omega(t)\in A^{0,q}(X,\Omega^n(\log D))$  such that
\begin{align}\label{2.25}
\b{\p}_t (e^{i_\varphi}(\Omega(t)))=0,\quad \Omega(0)=\Omega,	
\end{align}
where $\b{\p}_t$ denotes the $\b{\p}$-operator with respect to the complex structure of $X_t$. By Proposition \ref{prop2}, $e^{i_\varphi}(\Omega(t))\in A^{0,q}(X_t,\Omega_t^n(\log D_t))$ is a smooth extension of $\Omega$ with $(e^{i_\varphi}(\Omega(t)))(0)=\Omega$ and thus the difficulty here lies in $\b{\p}$-closedness of \eqref{2.25}.

 From \cite[Proposition 5.1]{Liu} or \cite[(2.14)]{RwZ}, (\ref{2.25}) is equivalent to
 \begin{align}\label{2.30}
 \b{\p}\Omega(t)+\p(\varphi\l\Omega(t))=0,\quad \Omega(0)=0.	
 \end{align}

We shall solve the equation \eqref{2.30} by an iteration method originally from \cite{LSY} and developed in \cite{Liu,RZ,RZ2, RZ15, RwZ}.
 To study the equation \eqref{2.30}, we need a logarithmic analogue of the Tian-Todorov lemma \cite{T87,To89}.
 \begin{lemma}\label{TianTodorov}
 	For any $\varphi, \psi\in A^{0,1}(X, T_X(-\log D))$ and $\Omega\in A^{0,q}(X,\Omega^n_X(\log D))$, we have
$$
 	[\varphi,\psi]\l \Omega=-\p(\psi\l(\varphi\l\Omega))+\psi\l\p(\varphi\l\Omega)+\varphi\l\p(\psi\l\Omega).	
$$
 \end{lemma}
\begin{proof}
Comparing this formula with \cite[Proposition 3.2]{Liu1}, one just needs to notice that this is a local formula from direct local computations, and for each $i=1,\cdots,r$,
$$\p\left(\frac{dz^i}{z^i}\right)=0.$$
\end{proof}

 Assuming that $\alpha(t)$ is a power
series of bundle-valued or logarithmic $(p,q)$-forms, expanded as
  $$\alpha(t)=\sum_{k=0}^{\infty}\sum_{i+j=k}\alpha_{i,j}t^i\b{t}^j,$$
one uses the notation
\[ \begin{cases}
\alpha(t) = \sum^{\infty}_{k=0} \alpha_k, \\[4pt]
\alpha_k = \sum_{i+j=k} \alpha_{i,j}t^i \overline{t}^j, \\
\end{cases} \]
where $\alpha_k$ is the $k$-degree homogeneous part in the expansion
of $\alpha(t)$ and all $\alpha_{i,j}$ are  bundle-valued or logarithmic
$(p,q)$-forms on $X_0$ with $\alpha(0)=\alpha_{0,0}$. Thus,  one will adopt the
notations
\begin{align}\label{2.26}
	\Omega(t)=\sum_{k=0}^{\infty}\Omega_k,\quad \varphi=\sum_{k=0}^{\infty}\varphi_k
\end{align}
with $\Omega_0=\Omega$, $\varphi_0=0$ and the integrability condition
$$
\b{\p}\varphi_k=\frac{1}{2}\sum_{i+j=k}[\varphi_i,\varphi_j].	
$$
 Substituting (\ref{2.26}) into (\ref{2.30}), one has
\begin{align}\label{2.27}
\b{\p}\Omega_k+\p\sum_{i+j=k}(\varphi_i\l \Omega_j)=0,\quad k\geq 0.	
\end{align}

Now we will solve the system of equations (\ref{2.27}) by induction. The step of $k=0$ is solved by $\Omega_0=\Omega$. Now assume that for all $k\leq l$, we have constructed $$\Omega_k\in A^{0,q}(X,\Omega^n_X(\log D)),$$
which satisfies (\ref{2.27}).

For $k=l+1$, one has
\begin{align}\label{3.8}
  \begin{split}
    &\quad-\b{\p}\sum_{i+j=l+1}\p(\varphi_i\l\Omega_j)\\
    &=\p\left(\sum_{i=1}^{l+1}\b{\p}\varphi_i\l\Omega_{l+1-i}+\sum_{i=1}^{l+1}\varphi_i\l\b{\p}\Omega_{l+1-i}\right)\\
    &=\p\left(\frac{1}{2}\sum_{i=1}^{l+1}\sum_{j=1}^{i-1}[\varphi_i,\varphi_{i-j}]\l\Omega_{l+1-i}-\sum_{i=1}^{l+1}\varphi_i\l\p\left(\sum_{j=1}^{l+1-i}\varphi_j\l\Omega_{l+1-i-j}\right)\right)\\
    &=\p\left(\frac{1}{2}\sum_{i=1}^{l+1}\sum_{j=1}^{i-1}\left(-\p(\varphi_j\l(\varphi_{i-j}\l\Omega_{l+1-i}))-\varphi_j\l(\varphi_{i-j}\l\p\Omega_{l+1-i})+\varphi_j\l\p(\varphi_{i-j}\l\Omega_{l+1-i})\right.\right.\\
    &\quad\quad\left.\left.+\varphi_{i-j}\l\p(\varphi_j\l\Omega_{l+1-i})\right)-\sum_{i=1}^{l+1}\varphi_i\l\p\left(\sum_{j=1}^{l+1-i}\varphi_j\l\Omega_{l+1-i-j}\right)\right)\\
    &=\p\left(\sum_{1\leq j<i\leq l+1}\varphi_j\l\p(\varphi_{i-j}\l\Omega_{l+1-i})-\sum_{i=1}^{l+1}\varphi_i\l\p\left(\sum_{j=1}^{l+1-i}\varphi_j\l\Omega_{l+1-i-j}\right)\right)\\
    &=0,
    \end{split}
  \end{align}
where the third equality follows from Lemma \ref{TianTodorov}.

By Theorem \ref{thm2} and (\ref{gamma}), there is a solution
\begin{align}\label{3.9}
	\Omega_{l+1}=-\mathcal{I}^{-1} \b{\p}_E^*\mb{G}''_E \mathcal{I}\left(\p\sum_{i+j=l+1}\varphi_i\l \Omega_j\right),
\end{align}
where $E:=\Omega^n(\log D)=K_X\otimes [D]$ and $ \b{\p}_E^*$ is the adjoint operator of $\b{\p}$ with respect to $(X,\omega)$ and $(E,h)$ for some Hermitian metric $h$. So
\begin{align*}
	\Omega(t)&=\Omega+\sum_{l=0}^{+\infty}\Omega_{l+1}
	=\Omega-\mathcal{I}^{-1} \b{\p}_E^*\mb{G}''_E \mathcal{I}\p\left(\varphi\l \Omega(t)\right),
\end{align*}
which is formally equivalent to
\begin{align}\label{formal solution}
\Omega(t)=\left(I+\mathcal{I}^{-1} \b{\p}_E^*\mb{G}''_E \mathcal{I}\p i_{\varphi}\right)^{-1}\Omega. 	
\end{align}

Next we will prove the formal solution (\ref{formal solution}) is smooth actually, which is a little more general than the result in \cite[Subsection 3.2]{RwZ}. Applying $\mathcal{I}$ to both sides of (\ref{formal solution}), one has
$$
\mathcal{I}(\Omega(t))=(I+ \b{\p}_E^*\mb{G}''_E\mathcal{I} \p i_{\varphi} \mathcal{I}^{-1})^{-1}\mathcal{I}(\Omega).	
$$

Fixing an integer $k\geq 2$ and a real $\alpha\in (0,1)$, we denote by $\|\cdot\|'_{k,\alpha}, \|\cdot\|_{k,\alpha}$ the H\"older norms of bundle-valued on $X$, $X\times \Delta$, respectively (cf. \cite[p. 275]{k}).

 For any $\eta(t)\in A^{0,q}(X\times \Delta,E)$, when restricted to each fiber of $\pi:X\times \Delta\to \Delta$, one has
\begin{align}\label{estimate}
\begin{split}
\| \b{\p}_E^*\mb{G}''_E\mathcal{I} \p i_{\varphi} \mathcal{I}^{-1}\eta(t)\|'_{k,\alpha}\leq C_{k,\alpha}\|\mathcal{I} \p i_{\varphi} \mathcal{I}^{-1}\eta(t)\|'_{k-1,\alpha}\leq C_{k,\alpha}C_1\|\mathcal{I} i_{\varphi}\mathcal{I}^{-1}\eta(t)\|'_{k,\alpha},
\end{split}
\end{align}
where the second inequality follows from \cite[Proposition 2.3]{MK}. Without loss of generality, we assume that $X\times \Delta=\cup_{i=1}^N U_i\times \Delta$, or one may shrink the $\Delta$ slightly.  Note that $ \b{\p}_E^*\mb{G}''_E\mathcal{I} \p\mathcal{I}^{-1}$ is an operator independent of $t$. By the definition of H\"older norm and (\ref{estimate}), one has
\begin{align}\label{estimate1}
\begin{split}
	\| \b{\p}_E^*\mb{G}''_E\mathcal{I} \p i_{\varphi} \mathcal{I}^{-1}\eta(t)\|_{k,\alpha}&=\| \b{\p}_E^*\mb{G}''_E\mathcal{I} \p \mathcal{I}^{-1}\circ \mathcal{I} i_{\varphi} \mathcal{I}^{-1}\eta(t)\|_{k,\alpha}\\
	&\leq C\sup_t\left(\sum_{h_1+h_2\leq k}\| \b{\p}_E^*\mb{G}''_E\mathcal{I} \p \mathcal{I}^{-1}\circ D^{h_1}_t\mathcal{I} i_{\varphi} \mathcal{I}^{-1}\eta(t)\|'_{h_2,\alpha}\right)\\
	&\leq C\sup_t\left(\sum_{h_1+h_2\leq k}\|D^{h_1}_t\mathcal{I} i_{\varphi} \mathcal{I}^{-1}\eta(t)\|'_{h_2,\alpha}\right)\\
	&\leq C\|\mathcal{I} i_{\varphi} \mathcal{I}^{-1}\eta(t)\|_{k,\alpha}\\
	&\leq C\|\varphi\|_{k,\alpha}\|\eta(t)\|_{k,\alpha},
\end{split}
\end{align}
where  $C=C(k,\alpha)$  is a constant  independent of $t$ and may change from line to line, and $D^{h_1}$ denotes the $h_1$-th differential operator on $t$-direction.

Thus, there is a constant $c_{k,\alpha}$ such that $C\|\varphi\|_{k,\alpha}\leq 1/2$ for $|t|<c_{k,\alpha}$.
From the estimate (\ref{estimate1}), one has
\begin{align}
\begin{split}
	\|\mathcal{I}(\Omega(t))\|_{k,\alpha}
	&\leq 2\left(\|\mathcal{I}(\Omega(t))\|_{k,\alpha}
	  -\| \b{\p}_E^*\mb{G}''_E\mathcal{I} \p i_{\varphi} \mathcal{I}^{-1}(\mathcal{I}(\Omega(t)))\|_{k,\alpha}\right)\\
	&\leq 2\|(I+ \b{\p}_E^*\mb{G}''_E\mathcal{I} \p i_{\varphi} \mathcal{I}^{-1})\mathcal{I}(\Omega(t))\|_{k,\alpha}\\
	&=2\|\mathcal{I}(\Omega)\|_{k,\alpha}.
	\end{split}
\end{align}
So $$\mathcal{I}(\Omega(t))=(I+ \b{\p}_E^*\mb{G}''_E\mathcal{I} \p i_{\varphi} \mathcal{I}^{-1})^{-1}\mathcal{I}(\Omega)$$ is convergent under $C^{k,\alpha}$-norm for $|t|<c_{k,\alpha}$.
Moreover,
\begin{align*}\label{2.31}
	&\quad \frac{\p^2}{\p t\p\b{t}}\mathcal{I}(\Omega(t))+\Delta^E_{\b{\p}}\mathcal{I}(\Omega(t))\\
&=\b{\p} \b{\p}_E^*\mathcal{I}(\Omega)- \b{\p}_E^*\mathcal{I} \p i_{\varphi}\Omega(t)- \b{\p}_E^*\mb{G}''_E\mathcal{I}\p\left(i_{\frac{\p \varphi}{\p t}}\frac{\p \Omega(t)}{\p \b{t}}+i_{\frac{\p \varphi}{\p \b{t}}}\frac{\p \Omega(t)}{\p t}+i_{\frac{\p^2 \varphi}{\p t\p\b{t}}}\Omega(t)+i_{\varphi}\frac{\p^2 \Omega(t)}{\p t\p \b{t}}\right).
\end{align*}
Since $\varphi(0)=0$, there exists a smaller uniform upper bound $c<c_{k,\alpha}$ such that the equation (\ref{2.31}) is a uniform fully nonlinear elliptic equation on $X\times \{|t|<c\}$. By \cite[Theorem 4.6]{Taylor}, $\mathcal{I}(\Omega(t))$ is smooth on $X\times \{|t|<c\}$ and so is $\Omega(t)$.

In summary, we have proved:
\begin{thm}\label{thm3}
Let $(X,\omega)$ be a compact K\"ahler manifold and $D$ be a simple normal crossing divisor on it. For any logarithmic deformations $(X_t,D_t), t\in S$ of pair $(X,D)$ with $X_{0}=X$, induced by $\varphi:=\varphi(t)\in A^{0,1}(X, T_X(-\log D))$,  and any $\b{\p}$-closed logarithmic $(n,q)$-form $\Omega$ on the central fiber $X$, there exists a small neighborhood $\Delta\subset S$ of $0$ and a smooth family $\Omega(t)$ of logarithmic $(n,q)$-form on the central fiber $X$, such that
\begin{align}
e^{i_\varphi}(\Omega(t))=e^{i_\varphi}\left(\left(I+\mathcal{I}^{-1} \b{\p}_E^*\mb{G}''_E \mathcal{I}\p i_{\varphi}\right)^{-1}\Omega\right)\in A^{0,q}(X_t,\Omega^n_{X_t}(\log D_t)),
\end{align}
which is $\b{\p}_t$-closed on $X_t$ for any $t\in \Delta$, and with $(e^{i_\varphi}(\Omega(t)))(0)=\Omega$.

\end{thm}

\begin{rem}
For the case of $D=\emptyset$ and for any $\b{\p}$-closed $\Omega\in A^{n,q}(X)$, one can get the solution (\ref{formal solution}) of the equation (\ref{2.30}) directly without using the expansion (\ref{2.26}) or  iteration formulas
(\ref{3.8}), (\ref{3.9}).

In fact, one only needs to check (\ref{formal solution}) satisfying (\ref{2.30}).
It follows from (\ref{formal solution}) and Lemma \ref{TianTodorov} that
\begin{align*}
\b{\p}\Omega(t)&=-\b{\p}\b{\p}^*\mb{G}\p i_{\varphi}\Omega(t)+\b{\p}\Omega\\
&=-\p i_{\varphi}\Omega(t)-\b{\p}^*\mb{G}\p i_{\varphi}\p i_{\varphi} \Omega(t)-\b{\p}^*\mb{G}\p i_{\varphi} \b{\p}\Omega(t)\\
&=-\p i_{\varphi}\Omega(t)-\b{\p}^*\mb{G}\p i_{\varphi}(\b{\p}+\p i_{\varphi})\Omega(t)\\
&=-\p i_{\varphi}\Omega(t)-\b{\p}^*\mb{G}\p i_{\varphi}(\b{\p}+\p i_{\varphi})\Omega(t).
\end{align*}
Therefore,
$$
(I+\b{\p}^*\mb{G}\p i_{\varphi})(\b{\p}+\p i_{\varphi})\Omega(t)=0.	
$$
By the invertibility of the operator $I+\b{\p}^*\mb{G}\p i_{\varphi}$, one has
$$
(\b{\p}+\p i_{\varphi})\Omega(t)=0.	
$$
\end{rem}

A pair $(X,D)$ is called a \textit{log Calabi-Yau pair} if the logarithmic canonical line bundle $\Omega^n_X(\log D)\cong \mc{O}_X(K_X+D)$ is trivial. By Theorem \ref{thm3}, one has
\begin{cor}\label{CYlog}
The log Calabi-Yau structure is locally stable, i.e., if the reference pair $(X,D)$ is a log Calabi-Yau pair, then there exists a small neighborhood  $\Delta\subset S$ of the reference point such that $(X_t,D_t), t\in \Delta$ is also a log Calabi-Yau pair for any family of logarithmic deformations $\mathscr{F}=(X_t, D_t), t\in S$.
\end{cor}
\begin{proof}
	By the definition of log Calabi-Yau structure, let $\mathcal{I}(\Omega)$ be the trivial section of $K_X\otimes [D]$, i.e., there are no zero points for the holomorphic section $\mathcal{I}(\Omega)\in H^0(X,K_X\otimes [D])$. By Theorem \ref{thm3}, the section
	$$
	\mathcal{I}(e^{i_\varphi}(\Omega(t)))\in H^0(X_t,K_{X_t}\otimes [D_t])	
	$$
 also has no zero points for small $t$. Thus, $(X_t, D_t)$ is a log Calabi-Yau pair.
\end{proof}

Denoting the logarithmic Hodge numbers by $h^{p,q}(X,D):=\dim_{\mathbb{C}} H^q(X,\Omega^p_X(\log D))$, we have the invariance of logarithmic Hodge numbers under small logarithmic deformations.
\begin{cor}
	The logarithmic Hodge numbers $h^{p,q}(X_t,D_t)$ are invariant under small logarithmic deformations.
\end{cor}
\begin{proof} Compare with \cite{Del71,Sn}.
	 By Kodaira-Spencer's upper semi-continuity of $h^{p,q}(X_t, D_t)$ (cf. \cite[Theorem 4.3]{MK}),
$$
	h^{p,q}(X_t, D_t)=\dim_{\mathbb{C}} H^{0,q}_{\b{\p}_t}(X_t, E^p_t)\leq \dim_{\mathbb{C}} H^{0,q}_{\b{\p}}(X, E^p)=h^{p,q}(X,D)
$$
 for small $t$. By Corollary \ref{decomposition}, one has
$$
 	\sum_{p+q=k}h^{p,q}(X,D)=\dim_{\mathbb{C}} H^k(X-D,\mb{C})=\sum_{p+q=k}h^{p,q}(X_t,D_t).
$$
 So $h^{n,q}(X, D)=h^{n,q}(X_t, D_t)$ for small $t$.
\end{proof}

In the next two subsections, we will use Theorems $\ref{thm2}$ and \ref{sol-CY} to prove two logarithmic deformation unobstructedness theorems by
Katzarkov-Kontsevich-Pantev \cite{KKP08} and Iacono \cite{Iacono} on K\"ahler manifolds. Our differential geometric proofs are quite different from theirs and have interesting applications to extension problems.
\subsection{Deformations of log Calabi-Yau pairs}

In this subsection, we will prove that the logarithmic deformation of a log Calabi-Yau pair is unobstructed by a purely differential geometric method. Recall that $(X,D)$ is a \emph{log Calabi-Yau pair} if $X$ is a compact K\"ahler manifold and $D$ is a simple normal crossing divisor with trivial $\Omega^n_X(\log D)=K_X\otimes [D]$. For any $[\varphi_1]\in H^{0,1}(X, T_X(-\log D))$ and any $t$ in a small $\epsilon$-disk $D_\epsilon$ of $0$ in $\mathbb{C}^{\dim_{\mathbb{C}} H^{0,1}(X, T_X(-\log D))}$, we try to construct a holomorphic family
$$\varphi:=\varphi(t)\in A^{0,1}(X, T_X(-\log D))$$ satisfying the following integrability and initial conditions:
\begin{align}\label{log 0.1}
\b{\p}\varphi=\frac{1}{2}[\varphi,\varphi],\quad \frac{\p \varphi}{\p t}(0)=\varphi_1.	
\end{align}

To solve the above equation, we need:
\begin{lemma}\label{log lemma2}
	Let $\Omega'\in A^{0,0}(X,\Omega^n_X(\log D))$ be a logarithmic $(n,0)$-form without zero points. Then
	$$
	\bullet\l\Omega': A^{0,1}(X,T_X(-\log D))\to A^{0,1}(X, \Omega^{n-1}_X(\log D))	
	$$
is an isomorphism, whose inverse, defined by $(\ref{log 0.8})$, is
$$
\Omega'^*\l\bullet: A^{0,1}(X, \Omega^{n-1}_X(\log D))\to A^{0,1}(X, T_X(-\log D)).
$$
\end{lemma}
\begin{proof}
Locally, we may assume that $D=\{z^1\cdots z^k=0\}$ and write $\Omega'$ as
\begin{align}\label{u}
	\Omega'=u\frac{dz^1\wedge\cdots\wedge dz^n}{z^1\cdots z^k},
\end{align}
	where $u$ is a smooth function on $X$ and admits no zero points. Any element $\varphi$ of $A^{0,1}(X, T_X(-\log D))$ is of the form in local coordinates,
$$
\varphi=\sum_{i=1}^k \varphi_{\b{j}}^i d\b{z}^j\otimes z^i\frac{\p\ }{\p z^i}+\sum_{i=k+1}^n \varphi_{\b{j}}^i d\b{z}^j\otimes \frac{\p\ }{\p z^i}.
$$
 Thus,
 \begin{align}\label{log 0.7}
 \begin{split}
 \varphi\l\Omega'&=u\sum_{i=1}^k\varphi^i_{\b{j}}d\b{z}^j\wedge(-1)^{i-1}\frac{dz^1\wedge \cdots \wedge\widehat{dz^i}\wedge\cdots\wedge dz^n}{z^1\cdots\widehat{z^i}\cdots z^k}\\
 &\quad +u\sum_{i=k+1}^n\varphi^i_{\b{j}}d\b{z}^j\wedge(-1)^{i-1}\frac{dz^1\wedge \cdots \wedge\widehat{dz^i}\wedge\cdots\wedge dz^n}{z^1\cdots z^k}.
 \end{split}
 \end{align}
 If $\varphi\l\Omega'=0$, then it follows that all coefficient functions $\varphi^i_{\b{j}}=0$ from (\ref{log 0.7}). Any $\psi\in A^{0,1}(X, \Omega_X^{n-1}(\log D))$ is locally
 \begin{align*}\label{log 0.8}
\psi&=\sum_{i=1}^k\psi^i_{\b{j}}d\b{z}^j\wedge(-1)^{i-1}\frac{dz^1\wedge \cdots \wedge\widehat{dz^i}\wedge\cdots\wedge dz^n}{z^1\cdots\widehat{z^i}\cdots z^k}\\
 &\quad +\sum_{i=k+1}^n\psi^i_{\b{j}}d\b{z}^j\wedge(-1)^{i-1}\frac{dz^1\wedge \cdots \wedge\widehat{dz^i}\wedge\cdots\wedge dz^n}{z^1\cdots z^k}.
\end{align*}
 One can define
 \begin{equation}\label{log 0.8}
\begin{aligned}
\Omega'^*\l&=\frac{1}{u}\sum_{i=1}^k (-1)^{n+i}i_{\frac{\p\ }{\p z^n}}\circ\cdots\widehat{i_{z^i\frac{\p\ }{\p z^i}}}\cdots\circ i_{z^1\frac{\p\ }{\p z^1}}\otimes z^i\frac{\p\ }{\p z^i}\\
&\quad +\frac{1}{u}\sum_{i=k+1}^n (-1)^{n+i}i_{\frac{\p\ }{\p z^n}}\circ\cdots\widehat{i_{\frac{\p\ }{\p z^i}}}\cdots\circ i_{z^1\frac{\p\ }{\p z^1}}\otimes \frac{\p\ }{\p z^i}.	
\end{aligned}
\end{equation}
Thus,
\begin{align*}
\varphi&:=\Omega'^*\l\psi=\frac{1}{u}\sum_{i=1}^k\psi^i_{\b{j}}d\b{z}^j\otimes z^i\frac{\p\ }{\p z^i}+\frac{1}{u}\sum_{i=k+1}^n\psi^i_{\b{j}}d\b{z}^j\otimes \frac{\p\ }{\p z^i}\\
&\in A^{0,1}(X,T_X(-\log D)).
\end{align*}
Moreover,
$$
(\Omega'^*\l\psi)\l\Omega'=\psi,\quad \Omega'^*\l(\varphi\l\Omega')=\varphi.	
$$
\end{proof}

Since $\Omega^n_X(\log D)$ is trivial, one may take a holomorphic logarithmic $(n,0)$-form $\Omega$ without zero points.
\begin{prop}\label{prop3}
If there are two smooth families $$\varphi(t)\in A^{0,1}(X, T_X(-\log D))$$ and $$\Omega(t)\in A^{0,0}(X,\Omega^{n}_X(\log D))$$ satisfying the system of equations
\begin{equation}\label{log 0.2}
\begin{cases}
 (\b{\p}+\frac{1}{2}\p\circ i_{\varphi})(i_{\varphi}\Omega(t))=0,\\
(\b{\p}+\p\circ i_{\varphi})\Omega(t)=0,\\
\Omega_0=\Omega,
\end{cases}
\end{equation}
then $\varphi(t)$ satisfies $(\ref{log 0.1})$ for sufficiently small $t$.
\end{prop}
\begin{proof}
From (\ref{log 0.2}) and Lemma \ref{TianTodorov}, one has
\begin{align}\label{log 0.3}
\begin{split}
\b{\p}(\varphi\l\Omega(t))&=-\frac{1}{2}\p\circ i_{\varphi}\circ i_{\varphi}\Omega(t)\\
&=\frac{1}{2}[\varphi,\varphi]\l\Omega(t)-i_{\varphi}\circ\p\circ i_{\varphi}\Omega(t)\\
&=\frac{1}{2}[\varphi,\varphi]\l\Omega(t)+i_{\varphi}\circ\b{\p}\Omega(t).
	\end{split}
\end{align}
Therefore,
$$
(\b{\p}\varphi)\l\Omega(t)=\b{\p}(\varphi\l\Omega(t))-i_{\varphi}\circ\b{\p}\Omega(t)=\frac{1}{2}[\varphi,\varphi]\l\Omega(t).
$$
Since $\Omega(t)$ is smooth and $\Omega(0)=\Omega_0=\Omega$, $\Omega(t)\in A^{0,0}(X, \Omega^n_X(\log D))$ also has no zero points for small $t$. One has from Lemma \ref{log lemma2}
$$
\b{\p}\varphi=\frac{1}{2}[\varphi,\varphi].	
$$
\end{proof}
To study the system of equations (\ref{log 0.2}), we need a logarithmic analogue of \cite[Lemma 4.2]{Liu}.
\begin{lemma}\label{log lemma1}
Assume that for $\varphi_{\nu}\in A^{0,1}(X, T_X(-\log D))$, $\nu=2,\ldots,K$,
$$
\b{\p}\varphi_{\nu}=\frac{1}{2}\sum_{\alpha+\beta=\nu}[\varphi_{\alpha},\varphi_{\beta}], \quad \b{\p}\varphi_1=0.	
$$
	Then one has
$$
	\b{\p}\left(\sum_{\nu+\gamma=K+1}[\varphi_{\nu},\varphi_{\gamma}]\right)=0.	
$$

\end{lemma}
\begin{proof}
  The local calculation is exactly the same as that in the proof of \cite[Lemma 4.2]{Liu} since one can regard $T_X(-\log D)$ as a subbundle of  $T_X$.
\end{proof}

Now we will solve the system of equations (\ref{log 0.2}) by mixing the methods originally from \cite{T87,To89,LSY} and developed in \cite{Liu,RZ,RZ2, RZ15, RwZ}.

Firstly, we denote $(\ref{log 0.2})_k$ by the system of equations of the first equation in (\ref{log 0.2}) with the $(k+1)$-th degree equal to $0$ and the second equation in (\ref{log 0.2}) with the $k$-th degree equal to $0$.

For $k=1$, the equation $(\ref{log 0.2})_{k=1}$ is
$$
\begin{cases}
	\b{\p}(i_{\varphi_2}\Omega_0+i_{\varphi_1}\Omega_1)+\frac{1}{2}\p\circ i_{\varphi_1}(i_{\varphi_1}\Omega_0)=0,\\
\b{\p}\Omega_1+\p\circ i_{\varphi_1}\Omega_0=0.
\end{cases}
$$
Thus, one can take
$$
 	\Omega_1=\mathcal{I}^{-1} \b{\p}_E^*\mb{G}''_E\mathcal{I}(-\p\circ i_{\varphi_1}\Omega_0)\in A^{0,0}(X, \Omega^n_X(\log D))
$$
and from Lemma \ref{log lemma2}, one can find $\varphi_2\in A^{0,1}(X, T_X(-\log D))$ satisfying
\begin{align*}
i_{\varphi_2}\Omega_0&=-i_{\varphi_1}\Omega_1+\mathcal{I}^{-1} \b{\p}_E^*\mb{G}''_E\mathcal{I}\left(-\frac{1}{2}\p\circ i_{\varphi_1}(i_{\varphi_1}\Omega_0)\right)\\
&\in A^{0,1}(X,\Omega^{n-1}_X(\log D))
\end{align*}
since $\varphi_1\l: A^{0,q}(X,\Omega^p_X(\log D))\to A^{0,q+1}(X,\Omega^{p-1}_X(\log D))$.

By induction, we may assume that the equation (\ref{log 0.2}) is solved for $k\leq l$ and we have constructed $\Omega_i, 0\leq i\leq l$ and $\varphi_i, 1\leq i\leq l+1$. So from (\ref{log 0.3}), we have
\begin{align}\label{log 0.4}
	\b{\p}\varphi_k=\frac{1}{2}\sum_{i+j=k}[\varphi_i,\varphi_j]\quad k\leq l+1.
\end{align}
Combining it with Lemma \ref{log lemma1}, one has
\begin{align}\label{log 0.5}
\b{\p}\left(\sum_{i+j=k}[\varphi_i,\varphi_j]\right)=0,\quad k\leq l+2.
\end{align}

Now for the $(l+1)$-th step, (\ref{log 0.2}) becomes
$$\label{log 0.6}
\begin{cases}
	\sum_{i+j=l+2}\b{\p}i_{\varphi_i}\Omega_j+\sum_{i+j+k=l+2}\frac{1}{2}\p\circ i_{\varphi_i}i_{\varphi_j}\Omega_k=0,\\
\b{\p}\Omega_{l+1}+\sum_{i+j=l+1}\p\circ i_{\varphi_i}\Omega_j=0.
\end{cases}
$$

Then we take
\begin{align}\label{Omega}
\Omega_{l+1}=\sum_{i+j=l+1}\mathcal{I}^{-1} \b{\p}_E^*\mb{G}''_E\mathcal{I}(-\p\circ i_{\varphi_i}\Omega_j)\in A^{0,0}(X, \Omega^n_X(\log D)).
\end{align}
On the other hand,
\begin{align*}
&\quad \b{\p}\sum_{i+j+k=l+2}\frac{1}{2}\p\circ i_{\varphi_i}\circ i_{\varphi_j}\circ \Omega_k\\
& =\sum_{i+j+k=l+2}\b{\p}\left(-[\varphi_j,\varphi_i]\l\Omega_k+\varphi_i\l\p(\varphi_j\l\Omega_k)+\varphi_j\l\p(\varphi_i\l\Omega_k)\right)\\
&=-\sum_{i+j+k=l+2}\b{\p}([\varphi_j,\varphi_i]\l\Omega_k)-2\sum_{j+k=l+2,k\geq 1}\b{\p}(\varphi_j\l\b{\p}\Omega_k)\\
&=-\sum_{k=0}^{l+2}\left(\b{\p}\sum_{i+j=l+2-k}[\varphi_j,\varphi_i]\right)\l\Omega_k+2\sum_{k=1}^{l+2}\left(\frac{1}{2}\sum_{i+j=l+2-k}[\varphi_j,\varphi_i]-\b{\p}\varphi_{l+2-k}\right)\l\b{\p}\Omega_k\\
&=0,
\end{align*}
where the last equality follows from (\ref{log 0.5}), (\ref{log 0.4}) and $\b{\p}\Omega_0=0$. From Lemma \ref{log lemma2}, we may take $\varphi_{l+2}\in A^{0,1}(X, T_X(-\log D))$ satisfying
\begin{align}\label{log varphi}
\begin{split}
i_{\varphi_{l+2}}\Omega_0&=-\sum_{j=1}^{l+2}i_{\varphi_{l+2-j}}\Omega_j+\mathcal{I}^{-1} \b{\p}_E^*\mb{G}''_E\mathcal{I}\left(-\sum_{i+j+k=l+2}\frac{1}{2}\p\circ i_{\varphi_i}i_{\varphi_j}\Omega_k\right)\\
&\in A^{0,1}(X,\Omega^{n-1}_X(\log D)).
\end{split}
\end{align}

 In one word, we have constructed $\Omega_i\in A^{0,0}(X,\Omega^n_X(\log D)), 0\leq i\leq l+1$ and $\varphi_i\in A^{0,1}(X, T_X(-\log D)), 1\leq i\leq l+2$, which satisfying $(\ref{log 0.2})_{l+1}$. So we can
solve (\ref{log 0.2}) inductively for small $t$.

Now we  prove the convergence of $\varphi(t)$ and $\Omega(t)$ under $C^{k,\alpha}$-norm. As in \cite[p. 50]{MK}, we may consider the power series
$$
A(t)=\frac{b}{16c}\sum_{m=1}^{\infty}\frac{(ct)^m}{m^2}=\sum_{m=1}^{\infty}a_m t^m,	\quad a_m=\frac{bc^{m-1}}{16m^2},
$$
which satisfies $$A^n(t)\leq (b/c)^{n-1}A(t)$$ and converges for $|t|<1/c$, where $b,c>0$.

 Fix an integer $k\geq 2$ and a real constant $\alpha\in (0,1)$.
 Suppose that they are chosen so that $\|\varphi_i\|_{k,\alpha}\leq a_i$ for $1\leq i\leq l+1$, and $\|\Omega_i\|_{k,\alpha}\leq (c/b)^{1/2}a_i$ for $1\leq i\leq l$. From (\ref{Omega}), it follows that
\begin{align}\label{log 1.2}
\begin{split}
\|\Omega_{l+1}\|_{k,\alpha}	&\leq C\sum_{i+j=l+1}\|\varphi_i\|_{k,\alpha}\|\Omega_j\|_{k,\alpha}\\
&\leq C\left(\left(\frac{c}{b}\right)^{1/2}\sum_{i+j=l+1}a_i a_j+\|\Omega_0\|_{k,\alpha}a_{l+1}\right)\\
&\leq C \left(\frac{b}{c}+\left(\frac{b}{c}\right)^{1/2}\|\Omega_0\|_{k,\alpha}\right)\left(\frac{c}{b}\right)^{1/2}a_{l+1}.
\end{split}
\end{align}

Also by (\ref{log varphi}), one has
$$
\varphi_{l+2}=-\Omega_0^*\l\left(\sum_{j=1}^{l+1}i_{\varphi_{l+2-j}}\Omega_j+\mathcal{I}^{-1} \b{\p}_E^*\mb{G}''_E\mathcal{I}\left(\sum_{i+j+k=l+2}\frac{1}{2}\p\circ i_{\varphi_i}i_{\varphi_j}\Omega_k\right)	\right).
$$
From the above expression, it follows
\begin{align}\label{log 1.1}
\begin{split}
\|\varphi_{l+2}\|_{k,\alpha}&\leq C\left(\left(\frac{c}{b}\right)^{1/2}\frac{b}{c}+\|\Omega_0\|_{k,\alpha}\frac{b}{c}+\left(\frac{c}{b}\right)^{1/2}\left(\frac{b}{c}\right)^2\right)a_{l+2}\\
&\leq C\left(\left(\frac{b}{c}\right)^{1/2}+\|\Omega_0\|_{k,\alpha}\frac{b}{c}+\left(\frac{b}{c}\right)^{3/2}\right) a_{l+2}.
\end{split}
\end{align}
Now one may take $b/c$  so small that
$$
\begin{cases}
	C \left(\frac{b}{c}+\left(\frac{b}{c}\right)^{1/2}\|\Omega_0\|_{k,\alpha}\right)\leq 1,\\
	C\left(\left(\frac{b}{c}\right)^{1/2}+\|\Omega_0\|_{k,\alpha}\frac{b}{c}+\left(\frac{b}{c}\right)^{3/2}\right)\leq 1.
\end{cases}
$$
By (\ref{log 1.1}) and (\ref{log 1.2}), one has
$$\|\varphi_{l+2}\|_{k,\alpha}\leq a_{l+2}$$ and $$\|\Omega_{l+1}\|_{k,\alpha}\leq (c/b)^{1/2}a_{l+1}.$$
 Since $b/c$ is invariant by  the same scaling to $c$ and $b$, we can assume $\|\varphi_1\|_{k,\alpha}\leq  a_1$.
Thus,
$$
\|\varphi(t)\|_{k,\alpha}\leq A(t),\quad \|\Omega(t)\|_{k,\alpha}\leq \|\Omega_0\|_{k,\alpha}+\left(\frac{c}{b}\right)^{1/2}A(t)
$$ for $|t|<1/c$.

Finally we come to the regularity argument of $\varphi:=\varphi(t)$, which is a little more difficult than that in \cite[Subsection 3.2]{RwZ} since one has to consider the regularity of $\varphi$ and $\Omega(t)$ by a simultaneous induction here. From (\ref{log varphi}) and (\ref{Omega}), one has
\begin{equation}\label{log 1.3}
\begin{cases}
	 (I+\mathcal{I}^{-1} \b{\p}_E^*\mb{G}''_E\mathcal{I} \p i_{\varphi})\Omega(t)=\Omega_0\\
	(I+\frac{1}{2}\mathcal{I}^{-1} \b{\p}_E^*\mb{G}''_E\mathcal{I} \p i_{\varphi})i_{\varphi}\Omega(t)=i_{\varphi_1}\Omega_0.
\end{cases}	
\end{equation}

From the first equation of \eqref{log 1.3}, one has $\frac{\p\Omega(t)}{\p\b{t}}=0$. So
\begin{align}\label{log 1.5}
\begin{cases}
\frac{\p^2}{\p t\p\b{t}}\mathcal{I}(\Omega(t))+\Delta^E_{\b{\p}}\mathcal{I}(\Omega(t))=
	\b{\p} \b{\p}_E^*\mathcal{I}(\Omega)- \b{\p}_E^*\mathcal{I} \p i_{\varphi}\Omega(t),\\
	\frac{\p^2}{\p t\p\b{t}}\mathcal{I}(i_{\varphi}\Omega(t))+\Delta_{\b{\p}}^E\mathcal{I}(i_{\varphi}\Omega(t))=\Delta_{\b{\p}}^E\mathcal{I}(i_{\varphi_1}\Omega_0)-\frac{1}{2} \b{\p}_E^*\mathcal{I}\p i_{\varphi} i_{\varphi}\Omega(t).
	\end{cases}
	\end{align}

Since $\varphi$ and $\Omega(t)$ are convergent under $C^{k,\alpha}$ norm,
\begin{align}\label{log 1.8}
\|\mathcal{I}(i_{\varphi}\Omega(t))\|_{k,\alpha}<C.	
\end{align}
For the second equation of (\ref{log 1.5}), one has
\begin{align}\label{log 1.6}
 \b{\p}_E^*\mathcal{I}(i_{\varphi}\Omega(t))= \b{\p}_E^*\mathcal{I}(i_{\varphi_1}\Omega_0).	
\end{align}
Noticing that $[ \b{\p}_E^*, \mathcal{I} \p\mathcal{I}^{-1}]$ is an operator of first order and (\ref{log 1.6}), we have
\begin{align*}
	\| \b{\p}_E^*\mathcal{I} \p i_{\varphi}\Omega(t)\|_{k-1,\alpha} &=\|[ \b{\p}_E^*,\mathcal{I} \p \mathcal{I}^{-1}] \mathcal{I} (i_{\varphi}\Omega(t))-\mathcal{I}\p\mathcal{I}^{-1} \b{\p}_E^*\mathcal{I}(i_{\varphi}\Omega(t))\|_{k-1,\alpha}\\
	&\leq C\|\mathcal{I} (i_{\varphi}\Omega(t))\|_{k,\alpha}+\|\mathcal{I}\p\mathcal{I}^{-1} \b{\p}_E^*\mathcal{I}(i_{\varphi_1}\Omega_0)\|_{k-1,\alpha}\leq C.
\end{align*}

By the first equation of (\ref{log 1.5}), one gets $\|\mathcal{I}(\Omega(t))\|_{k+1, \alpha}<C$. By the expression (\ref{u}), the local function $u(t)$ associated with $\Omega(t)$ is locally in $C^{k+1,\alpha}$, i.e., for each $U_i$, $|u(t)|^{U_i}_{k+1,\alpha}<C$ (cf. \cite[275]{k}). By shrinking $\{|t|<c\}$ slightly smaller, we may assume that $|u(t)|^{U_i}_0>c_0$ since $u(0)$ has no zero points. By a direct computation, one has
\begin{align}\label{log 1.10}
	\left|\frac{1}{u(t)}\right|^{U_i}_{k+1,\alpha}<C.
\end{align}

From (\ref{log 0.8}) and (\ref{log 1.10}), we can view
\begin{align}\label{log 1.7}
	-\frac{1}{2} \b{\p}_E^*\mathcal{I}\p i_{\varphi} i_{\varphi}\Omega(t)=-\frac{1}{2} \b{\p}_E^*\mathcal{I}\p (\Omega(t)^*\l i_{\varphi}\Omega(t))\l i_{\varphi}\Omega(t)
\end{align}
 as a linear operator of second order in terms of $\mathcal{I}(i_{\varphi}\Omega(t))$ with $C^{k-1,\alpha}$-coefficients.
By (\ref{log 1.7}) and $\varphi(0)=0$, the second  equation of (\ref{log 1.5}) can be viewed as a linear elliptic equation with $C^{k-1,\alpha}$-coefficients when $t$ is small enough. From \cite[Theorem 6.17]{Gilbarg}, it follows that
\begin{align}\label{log 1.9}
	\|\mathcal{I}(i_{\varphi}\Omega(t))\|_{k+1,\alpha}<C.
\end{align}

By repeating the processes from (\ref{log 1.8}) to (\ref{log 1.9}), one gets the smoothness of $\mathcal{I}(\Omega(t))$, $\mathcal{I}(i_{\varphi}\Omega(t))$ and thus those of $\Omega(t)$ and $i_{\varphi}\Omega(t)$. So
$$
\varphi(t)=\Omega(t)^*\l(i_{\varphi}\Omega(t))
$$
is smooth on $X\times \{|t|<\epsilon\}$ for some $0<\epsilon<c$ inductively.

In conclusion, we get
\begin{thm}\label{Bel-equ}
Let $[\varphi_1]\in H^{0,1}(X, T_X(-\log D))$.
On a small $\epsilon$-disk of $0$ in $\mathbb{C}^{\dim_{\mathbb{C}} H^{0,1}(X, T_X(-\log D))}$, one can construct a holomorphic family $$\varphi(t)\in A^{0,1}(X, T_X(-\log D)),$$  such that
$$
\b{\p}\varphi(t)=\frac{1}{2}[\varphi(t),\varphi(t)],  \quad \frac{\p \varphi}{\p t}(0)=\varphi_1.	
$$
\end{thm}

\subsection{Logarithmic deformations on Calabi-Yau manifolds}
Let $X$ be a \emph{Calabi-Yau manifold}, i.e.,  an $n$-dimensional K\"ahler manifold admitting a nowhere vanishing holomorphic $(n,0)$-form, and $D$ a simple normal crossing divisor on it.
It is well-known that the deformations of a Calabi-Yau manifold are unobstructed, by the Bogomolov-Tian-Todorov theorem, due to F. Bogomolov \cite{B}, G. Tian \cite{T87} and A. Todorov \cite{To89}.
In this subsection, we prove that the logarithmic deformation of a pair on a Calabi-Yau manifold is unobstructed by an analogous method to Theorem \ref{Bel-equ} by constructing a family of integrable logarithmic Beltrami differentials on a small disk.

\begin{thm}\label{log}
With the above setting, the pair $(X,D)$ has unobstructed logarithmic deformations.
\end{thm}
\begin{proof}
  Since there is a holomorphic $(n,0)$-form $\Omega$ without zero points on $X$, any $(n,0)$-form $\Omega(t)$ of $X$ smoothly depending  on $t$  with $\Omega(0)=\Omega$ still admits no zero points for small $t$. As reasoned in Lemma \ref{log lemma2}, there holds the isomorphism
$$
\bullet\l\Omega(t): A^{0,1}(X, T_X(-\log D))\to A^{n,1}(X, T_X(-\log D))\subset A^{n-1,1}(X)
$$
with its inverse $$\Omega(t)^*\l:  A^{n,1}(X, T_X(-\log D))\to A^{0,1}(X, T_X(-\log D)).$$
By the same argument as in Proposition \ref{prop3}, one only needs to solve the system of equations
$$
\begin{cases}
(\b{\p}+\frac{1}{2}\p\circ i_{\varphi})(i_{\varphi}\Omega(t))=0,\\
(\b{\p}+\p\circ i_{\varphi})\Omega(t)=0.
\end{cases}
$$
By the same argument as the log Calabi-Yau pair and Theorem \ref{sol-CY}, for any $$[\varphi_1]\in H^{0,1}(X, T_X(-\log D)),$$
we can construct a holomorphic family
$$\varphi(t)\in A^{0,1}(X, T_X(-\log D))$$
on a small $\epsilon$-disk of $0$ in $\mathbb{C}^{\dim_{\mathbb{C}} H^{0,1}(X, T_X(-\log D))}$, satisfying the integrability and initial conditions:
$$
\b{\p}\varphi(t)=\frac{1}{2}[\varphi(t),\varphi(t)], \quad \frac{\p\varphi}{\p t}(0)=\varphi_1.
$$
\end{proof}

\end{document}